\documentclass[10pt,pdftex]{article}
\usepackage[a4paper,top=2cm,bottom=2.5cm,left=2cm,right=2cm,includeheadfoot]{geometry}
\usepackage{amsmath}
\usepackage{amsfonts}
\usepackage{overpic}
\usepackage{amsthm}
\usepackage{amssymb}
\usepackage{mathrsfs}

\usepackage{subfigure}

\usepackage{xcolor}

\usepackage{lipsum}

\newtheorem{theorem}{Theorem}
\newtheorem{remark}{Remark}

\begin{document}
\title{Phase-field modeling and effective simulation of non-isothermal reactive transport}

\author{Carina Bringedal\footnote{Department of Computer science, Electrical engineering and Mathematical sciences, Western Norway University of Applied Sciences, and Institute for Modelling Hydraulic and Environmental Systems, University of Stuttgart. carina.bringedal@hvl.no}, Alexander Jaust\footnote{Institute for Parallel and Distributed Systems, University of Stuttgart}}
\date{}
\maketitle

\begin{abstract}
	We consider single-phase flow with solute transport where ions in the fluid can precipitate and form a mineral, and where the mineral can dissolve and release solute into the fluid. Such a setting includes an evolving interface between fluid and mineral.  We approximate the evolving interface with a diffuse interface, which is modeled with an Allen-Cahn equation. We also include effects from temperature such that the reaction rate can depend on temperature, and allow heat conduction through fluid and mineral.
	As Allen-Cahn is generally not conservative due to curvature-driven motion, we include a reformulation that is conservative. This reformulation includes a non-local term which makes the use of standard Newton iterations for solving the resulting non-linear system of equations very slow. We instead apply L-scheme iterations, which can be proven to converge for any starting guess, although giving only linear convergence. The three coupled equations for diffuse interface, solute transport and heat transport are solved via an iterative coupling scheme. This allows the three equations to be solved more efficiently compared to a monolithic scheme, and only few iterations are needed for high accuracy. Through numerical experiments we highlight the usefulness and efficiency of the suggested numerical scheme and the applicability of the resulting model.
\end{abstract}

\section{Introduction}
Reactive transport is relevant for many environmental and technical applications. We are here especially interested in reactive transport with heterogeneous reactions, where solutes can leave a liquid phase and become part of a solid phase, or the other way around. This is a typical feature of mineral precipitation and dissolution \cite{reactivereview}, which is common in groundwater remediation, soil salinization, oil production, metallurgy and geothermal energy production, to mention some. When temperature plays a role, which is in particular the case with geothermal energy production, the interplay between temperature-dependent solubilities of minerals and reaction rates depending on temperature can give an interplay between varying temperature and solute concentrations \cite{clauser2003numerical}. For many of these applications, experiments are unfeasible, which rather turns the attention to having reliable models and numerical simulations to investigate these processes. However, the close coupling between transport of solutes, heat and chemical reactions can be challenging to model and simulate\cite{xu1999modeling}.

At the interface between fluid and mineral, solutes in the fluid can leave the liquid fluid phase and rather become part of the solid mineral phase. The transition of solutes from fluid to mineral corresponds to an evolving fluid-mineral interface and hence a free boundary of the model. Hence, where the fluid is and where the mineral is, pose additional challenges to modeling of such reactive transport processes. 

There are several modeling approaches to describe these reactive transport processes. One possibility is to use sharp-interface models, which includes to track the interface between the fluid and mineral and hence the fluid and mineral domains themselves, and at the same time solve the necessary conservation equations for mass and energy in the respective, time-evolving domains. The two domains are coupled to each other through boundary conditions on the evolving fluid-mineral interface, where these boundary conditions account for the exchange of mass and energy between the domains. To model the evolution of the interface one can in the general case use a level-set equation \cite{osher2005level}, for applications of level-set models with reactive transport see e.g.~\cite{bringedal2016upscaling,ray2019numerical,li2008levelset}. For simpler geometries, e.g.~layering or circular shapes, a simple distance or radius function can be used \cite{Bringedal2015,noorden2009irc,vannoorden2009strip}. 
Although correct and accurate, these sharp-interface approaches lead to singularities in the formulation in case of merging or break-up of two minerals, and can be numerically challenging due to the need of tracking the interface and divide the computational domains accordingly. Also, the interface itself appears as a shock in the domain, with discontinuities across it, leading to further numerical difficulties. We instead apply a diffuse-interface approach through a phase-field model.

In a diffuse-interface model \cite{anderson1998diffuse}, the interface between fluid and mineral is modeled as a diffuse transition zone. Instead of having a sharp interface with boundary conditions and different model equations and material properties at each side the evolving interface, the model is rewritten to use the same model equations at the full domain, where a diffuse transition region represents the processes at the evolving interface. This diffuse transition region hence represents the fluid-mineral interface and should account for the exchange of mass and energy occurring here, but is modeled as smooth. Hence, a diffuse-interface model does not involve any shocks nor discontinuities as the transition between fluid and mineral is smoothed out over some region of non-zero width. The evolution of the fluid-mineral interface is modeled by an evolution of the location of the diffuse transition region, through a phase-field equation. However, note that the introduction of the diffuse transition region is a mathematical simplification and does not represent physics on the molecular level (as the real interface width would be at a much smaller scale \cite{jacqmin2000contact}). Hence, the diffuse-interface model should be seen as an approximation to the real sharp-interface physics, introduced for convenient mathematical and numerical treatment. If the diffuse-interface model correctly approximates the sharp-interface physics; that is, the diffuse-interface model can recover the sharp-interface model as the diffuse interface width approaches zero, we consider the diffuse-interface model a valid approximation. This can be shown analytically by matched asymptotic expansions \cite{caginalp1988dynamics}, or compared numerically for test cases \cite{kelm2022comparison,xu2012compare,bringedal2019phase}. The diffuse-interface model offers advantages of avoiding singularities in case of merging and break-up, and is generally easier to discretize and solve numerically as it is formulated on a fixed domain without any shocks and discontinuities due to the smooth, diffuse nature. However, these advantages come at the cost of applying a model that is an approximation of the sharp-interface physics, which can lead to unwanted evolution of the diffuse transition region.

There are two main approaches for phase-field models describing evolving interfaces: Allen-Cahn \cite{allen1979cahn} and Cahn-Hilliard \cite{cahn1958hilliard}. Where the Allen-Cahn model is represented by a second-order non-linear partial differential equation (PDE), the Cahn-Hilliard model is a fourth-order non-linear PDE. Due to being second-order, the Allen-Cahn equation can be numerically discretized with most standard discretization schemes, and fulfills a maximum principle. However, the Allen-Cahn model includes curvature-driven motion of the interface \cite{allen1979cahn}, which in its original form is not conservative. That means, an Allen-Cahn phase field would evolve in a manner that corresponds to loss or gain of mineral that cannot be connected to mineral precipitation or dissolution reactions. On the other hand, the Cahn-Hilliard equation needs special care due to being fourth order and can have overshoot/undershoot of values, which need to be handled to stay within the physical bounds \cite{frank2020bound}. The Cahn-Hilliard model includes Mullins-Sekerka motion \cite{stoth1996mullins}, which corresponds to a redistribution of the mineral to minimize surface area. However, this redistribution is conservative, which means that any net loss or gain of mineral is always connected to a mineral precipitation or dissolution reaction. Both approaches have been successfully used to model mineral precipitation and dissolution, coupled with fluid flow and solute transport \cite{bringedal2019phase,lars2021ternary}.

We here apply the Allen-Cahn equation due to its beneficial numerical properties and availability of maximum principle, and rather explore a reformulation of it to make the curvature-driven motion conservative. We build on the Allen-Cahn model found in \cite{bringedal2019phase}. Rubinstein and Sternberg \cite{rubinstein1992nonlocal} suggested to add a non-local term to overcome the issue of non-conservative motion. Then, the curvature-driven motion instead redistributes the mineral to obtain mean curvature \cite{chen2010cons}. This reformulation has been successfully applied in several contexts \cite{chen2010cons,bringedal2019FVCA,peszynska2021reduced,frank2018mass}. Although conservative, the non-local term introduces numerical challenges in terms of solving the resulting non-linear system of equations if implicit time stepping is used. If applying Newton's method, the non-local term causes the Jacobian matrix to be full, which makes the Newton iterations expensive \cite{bringedal2019FVCA}. 

To avoid expensive Newton iterations caused by the non-local term of the conservative Allen-Cahn equation, we instead apply L-scheme iterations to solve the resulting non-linear system of equations \cite{pop2004lscheme}. Using L-scheme iterations relies on having a monotonically decreasing (or increasing) non-linearity, which is generally not the case for the conservative Allen-Cahn equation. Therefore we split the non-linearity in a decreasing and increasing part, which is solved implicitly and explicitly, respectively. By adding a stabilization parameter, we can guarantee that the L-scheme iterations form a contraction and always converge \cite{kumar2014convergence}. The L-scheme iterations are found to converge when the stabilization parameter is large enough, and converge independently of the time-step size and initial guess. L-scheme iterations however only offer linear convergence, and can require many iterations depending on the choice of stabilization parameter and time-step size. For other equations approaches with using modified L-scheme iterations or switching to Newton when a good starting point is found, have been suggested \cite{BringedalBastidas3,mitra2019lscheme,list2016iterative}. However, it is important to emphasize that for L-scheme iterations, every iteration step is much cheaper compared to Newton iterations.

To discretize the model equations, we apply the Finite Volume method \cite{eymard2000finite}. Finite Volume method is discretely conservative and the conservative Allen-Cahn equation remains discretely conservative when a suitable splitting of the non-linearity into implicit and explicit parts are made \cite{bringedal2019FVCA}.

To be able to model non-isothermal reactive transport, the Allen-Cahn equation must be coupled to conservation equations for mass and energy. As discussed above, these conservation equations are in this case defined on fixed domains and hence describe the conservation of mass and energy in the combined domain of both fluid and mineral. Hence, the conservation equations are modified to incorporate the diffuse transition that represents the mass and energy transfer between fluid and mineral. For this, we build on the isothermal model in \cite{bringedal2019phase} and extend it to non-isothermal conditions. We show how the conservation property of these equations are modified by coupling to the (conservative) Allen-Cahn equation. 

To solve the coupled system of conservative Allen-Cahn equation, conservation of mass and energy, we apply an iterative strategy where the different equations are solved sequentially at each time step, and iterate until convergence is achieved. This is done for a simpler case without fluid flow only. The advantage of using an iterative strategy is that each equation can be discretized according to its needs. In particular, as the conservation equations for mass and energy are linear (when the phase field is known), simpler strategies can be used to solve them. By introducing an additional stabilization parameter for the coupling, we can prove the convergence of the scheme with a restriction on the time-step size. The strategy and convergence proof build on the ideas in \cite{BringedalBastidas3,brun2020iterative}. Numerical simulations show that the resulting scheme performs well, requiring generally few coupling iterations to reach convergence. For larger time-step sizes, the coupling iterations still converge but the efficiency of the scheme is limited by the L-scheme iterations then requiring many iterations to reach convergence. 

The goal of this paper is hence to formulate a conservative phase-field model including an evolving fluid-mineral interface that can model fluid flow, solute transport and heat transport. Given the numerical challenges coming from such a highly non-linear and strongly coupled model, we aim to formulate an efficient numerical scheme based on L-scheme iterations for the non-linearities and coupling iterations for the coupled system of equations.

To reach our goals, the remaining of this paper is structured as following. In Section \ref{sec:sharp} we consider the sharp-interface model that form the basis of the physical processes that we want to model. The corresponding diffuse-interface model is formulated and analyzed in terms of conservation in Section \ref{sec:diffuse}. In Section \ref{sec:limit} we show that the sharp-interface limit of the diffuse-interface model is indeed the model from Section \ref{sec:sharp}, with the additional curvature-driven motion. In Section \ref{sec:numerics} we formulate the numerical strategy and show the discrete conservation as well as prove convergence of the two iterative schemes. Numerical behavior of the Allen-Cahn equation and the coupled phase-field model is investigated in Section \ref{sec:numbehavior}. Discussion and conclusion can be found in Section \ref{sec:conclusion}.

\section{Sharp-interface model}\label{sec:sharp}
We here present the sharp-interface model for fluid flow, solute conservation and energy conservation in the fluid and mineral. The model is similar to the model considered in \cite{bringedal2016upscaling,vannoorden2009strip}. 
Two domains are included: The fluid domain where fluid flow and the dissolved solute is found, and the mineral domain where the solute appears as a stationary mineral. Both domains are evolving in time as mineral dissolves becoming solute part of the fluid, and as solute from the fluid precipitates becoming part of the mineral. We will denote the fluid and mineral domains $\Omega_f(t)$ and $\Omega_m(t)$. The two domains are separated by the interface $\Gamma(t)$. The considered chemical reaction can be written as 
\begin{equation*}
n_1 C_1 + n_2 C_2 \rightleftharpoons C;
\end{equation*}
that is, that $n_1$ ions of type $C_1$ go together with $n_2$ ions of type $C_2$ to form one minerals molecule of type $C$. The reaction can go both ways, depending on the ion concentrations relative to the minerals solubility.

\subsection{Fluid domain}
In the fluid domain, we consider a fluid having density $\rho_f$ and which is flowing with velocity $\mathbf v$. Dissolved in the fluid, two solutes having concentration $c_1$ and $c_2$ (corresponding to the two ion types $C_1$ and $C_2$, respectively) are transported. Temperature variations are accounted for, and the temperature of the fluid is denoted $T_f$. Since we will work with conservation of both fluid and solutes, it is convenient to also consider the molar concentration of the fluid $m_f$. Note however that $m_f$ denotes the molar concentration of the whole fluid phase; hence it accounts for both solvent (water) and solutes. 
The molar concentration and the density of the fluid is connected through $\rho_f = m_f M$, where $M_f$ is the (average) molar mass of the fluid. We assume $M_f$ as well as $\rho_f$ to be constant, which is a valid approximation when the variability of $c_1$ and $c_2$ is not too large compared to $m_f$, or when the solvent and solutes have comparable molar masses. This means that we can still model the fluid with constant density.

The model equations for the considered processes in the fluid domain $\Omega_f(t)$ are:
\begin{subequations}\label{eq:sharpfluid}
\begin{align}
\nabla\cdot\mathbf v &= 0\label{eq:sharpmass}\\
\partial_t(\rho_f \mathbf v) + \nabla\cdot(\rho_f\mathbf v\otimes\mathbf v) &= -\nabla p + \nabla\cdot\boldsymbol{\sigma} + \rho_f\mathbf g \label{eq:sharpns} \\
\partial_t c_i + \nabla\cdot(c_i\mathbf v) &= D_i\nabla^2 c_i \quad i=1,2 \label{eq:sharpion}\\
\partial_t(c_{p,f}\rho_f T_f) + \nabla\cdot(c_{p,f}\rho_fT_f\mathbf v) &= k_f\nabla^2T_f \label{eq:sharpenergy}
\end{align}
\end{subequations}
Here, since the fluid is incompressible, the stress tensor is $\boldsymbol{\sigma} = \mu(T_f)\nabla\mathbf v$, where
the viscosity $\mu$ is allowed to vary with fluid temperature. We account for a body force $\rho_f\mathbf g$, e.g. gravity. Further, $D_i$ is the diffusivity of the solutes and we will assume $D_1=D_2=D$ and that it is constant. The heat capacity $c_{p,f}$ and heat conductivity $k_f$ are also assumed constant. 
Note that the first equation (\ref{eq:sharpmass}) is the mass conservation equation for the incompressible fluid and is originally $\partial_t m_f + \nabla\cdot(m_f\mathbf v) = 0$, but is simplified to (\ref{eq:sharpmass}) as $m_f$ is constant. The second equation (\ref{eq:sharpns}) is the conservation of linear momentum, while the third equation (\ref{eq:sharpion}) is the mass conservation equation of the solute. Finally, (\ref{eq:sharpenergy}) expresses the conservation of energy inside the fluid. Note that we use a rather simple energy equation by letting temperature be the primary variable, but this is a valid assumption as long we consider low fluid velocities and no evaporation of the fluid \cite{landau1987fluid}.

\subsection{Mineral domain}\label{sec:sharpmineral}
The mineral domain consists of a stationary mineral having molar concentration $m_m$ and mass density $\rho_m$. These two are connected through $\rho_m = m_m M_m$, and are all considered constant. The mineral is a solid, hence flow cannot take place. The mineral can however conduct heat, hence an energy conservation equation is considered in the mineral domain $\Omega_m(t)$:
\begin{subequations}\label{eq:sharpmineral}
\begin{align}
\mathbf v &= \mathbf 0 \label{eq:sharpmineralflow}\\
\partial_t(c_{p,m}\rho_mT_m) &= k_m\nabla^2T_m, \label{eq:sharpmineralenergy}
\end{align}
\end{subequations}
where $T_m$ is the temperature in the mineral, and $c_{p,m}$ and $k_m$ are the (constant) heat capacity and heat conductivity of the mineral, respectively. Here, (\ref{eq:sharpmineralflow}) expresses conservation of mass in the mineral, while (\ref{eq:sharpmineralenergy}) represents conservation of energy in the mineral.

\subsection{Boundary conditions between the two domains}
The interface $\Gamma(t)$ between the fluid and mineral domains is evolving due to the mineral precipitation and dissolution reaction. When minerals dissolve into ions, the ions become part of the fluid domain as solute. Hence, there is a transfer of ions across the interface, which also corresponds to a mass transfer across the interface. The interface obtains a new location due to the lost minerals. The opposite takes place due to precipitation.
To ensure conservation of mass, ions and energy across the interface, we apply Rankine-Hugoniot jump conditions \cite{fasano2008RH}, which balances the jump in the flux with the jump of the conserved quantity across an evolving interface. For momentum we assume a no-slip condition; that is, that the tangential component of the fluid velocity is zero at the interface. The normal component of the fluid velocity will in general not be zero, but be proportional to the normal velocity of the evolving interface. Hence, on $\Gamma(t)$ we have:
\begin{subequations}\label{eq:sharpbc}
\begin{align}
-m_f\mathbf v\cdot\mathbf n &= v_n((n_1+n_2)m_m-m_f), \label{eq:sharpbcmass}\\
\mathbf v &= v_n\frac{m_f-(n_1+n_2)m_m}{m_f}\mathbf n, \label{eq:sharpbcmom} \\
(-c_i\mathbf v+D\nabla c_i)\cdot\mathbf n &= v_n(n_im_m-c_i) \quad i=1,2, \label{eq:sharpbcsol}\\
(-k_m\nabla T_m - c_{p,f}\rho_f T_f\mathbf v + k_f\nabla T_f)\cdot\mathbf n &= v_n(c_{p,m}\rho_m T_m - c_{p,f}\rho_f T_f). \label{eq:sharpbcT}
\end{align}
\end{subequations}
The normal vector $\mathbf n$ is assumed to point into the mineral, and the normal velocity $v_n$ is positive when $\Gamma(t)$ moves towards the mineral domain, which corresponds to dissolution taking place.
The top boundary condition (\ref{eq:sharpbcmass}) considers conservation of mass, from a molar perspective. It expresses that one mineral molecule from the mineral domain dissolves into $n_1+n_2$ ions in the fluid domain. This jump in mass is balanced by the jump in the mass flux, which comprises of the advective term from the mass conservation equation. The second boundary condition (\ref{eq:sharpbcmom}) is a reformulation of the no-slip condition and states that the normal component of the fluid velocity from (\ref{eq:sharpbcmass}) is the only component. The third boundary condition (\ref{eq:sharpbcsol}) ensures conservation of the two solutes: As one mineral molecule forms $n_i$ ion molecules, the factor $n_i$ appears. The jump in solute is balanced by the jump in the solute flux. As the mineral is stationary, only the contribution from the fluid domain is needed. These boundary conditions (\ref{eq:sharpbcmass})-(\ref{eq:sharpbcsol}) are the same as in \cite{vannoorden2009strip}. Finally, the last boundary condition (\ref{eq:sharpbcT}) accounts for the energy conservation across the evolving interface. In this equation the jump in energy is balanced by the jump in the energy flux, which has contributions from both mineral and fluid. This boundary condition is also found in \cite{Bringedal2015}.

\subsubsection{Reaction rates for mineral precipitation and dissolution}
For the mineral precipitation/dissolution reaction, we formulate a reaction rate. Following \cite{knabner1995compatible,vannoorden2009strip} and including temperature dependence as in \cite{bringedal2016upscaling} we get
\begin{equation}
m_m v_n
= (-f_p(T_f,c_1,c_2) - f_d(T_f,c_1,c_2,\mathbf x)), \quad \mathbf x\in\Gamma(t),
\label{eq:fsharp}
\end{equation} 
where $f_p$ and $f_d$ are the precipitation and dissolution rates, respectively. The rates are \cite{Chang,knabner1995compatible}
\begin{subequations}\label{eq:fpfd}
\begin{align}
f_p(T_f,c_1,c_2) &= k_0 e^{-\frac{E}{RT_f}}\frac{c_1^{n_1}c_2^{n_2}}{K_m(T_f)} \text{ and } \\ f_d(T_f,c_1,c_2,\mathbf x) &= k_0 e^{-\frac{E}{RT_f}} w(\text{dist}(\mathbf x,\mathcal G),T_f,c_1,c_2) \label{eq:fd}
\end{align}
\end{subequations}
where $k_0$ is a positive rate constant, $E$ is the activation energy, $R$ the gas constant and $K_m(T_f)$ is the solubility product for the mineral. As dissolution can only take place when there is still mineral present, we introduce a distance measure that measures the distance from $\Gamma(t)$ to the closest non-reactive solid or an external boundary, in (\ref{eq:fd}) denoted $\mathcal G$. When there is no mineral present at a specific location, this distance will be zero. To incorporate zero dissolution rate when there is no mineral present, one can use a modified Heaviside graph \cite{knabner1995compatible}
\begin{equation*}
w(d,T_f,c_1,c_2) = \left\{ \begin{array}{lr} 0 & \text{if } d<0,\\ \text{min}\{\frac{c_1^{n_1}c_1^{n_2}}{K_m(T_f)},1\} & \text {if } d=0, \\
	1 & \text{if } d>0,
	\end{array}
	\right.
\end{equation*}
The temperature can both affect the solubility of the mineral and also the general speed of the reaction rates.

\subsubsection{Volume changes due to mineral precipitation and dissolution}\label{sec:massbc}
The boundary condition ensuring conservation of mass (\ref{eq:sharpbcmass}) and the modified no-slip condition (\ref{eq:sharpbcmom}) also expresses any volume expansion occurring due to the chemical reaction. Letting $K_m=\frac{m_f-(n_1+n_2)m_m}{m_f}$, the boundary condition (\ref{eq:sharpbcmom}) on $\Gamma(t)$ reads
\begin{equation*}
\mathbf v= K_m v_n\mathbf n.
\end{equation*}
Hence, the fluid near the evolving interface is moving proportionally with the normal velocity of the evolving interface, and is scaled with the factor $K_m$. This is caused by the fluid having to move due to expansions/contractions caused by the chemical reaction. 

If $K_m=0$, then $m_f=(n_1+n_2)m_m$ which corresponds to the volume taken up by one mineral molecule being the same as the volume taken up by $n_1+n_2$ solute molecules. In this case, we have the more commonly used condition $\mathbf v=\mathbf 0$ on the interface. As the resulting velocity at the interface is generally small compared to flow velocities, it is usually a good approximation to let $K_m=0$ even if it is non-zero \cite{li2008levelset}. 

If $K_m>0$, the fluid density is larger which correspond to that the mineral phase takes up a larger volume compared to the fluid. This means, for example in the case of dissolution and $v_n>0$, that the fluid has to move towards the mineral to account for the volume change. Hence the fluid flows towards the mineral with the velocity $\mathbf v= K_m v_n\mathbf n$, which is positive in the direction of the normal vector pointing into the mineral.
Oppositely, if $K_m<0$, the fluid density is lower, hence the mineral phase takes up less volume compared to the fluid. In the case of dissolution and $v_n>0$, the fluid has to move away from the mineral as it is pushed away by the released solute needing more room. The fluid flows away from the mineral with the velocity $\mathbf v= K_m v_n\mathbf n$, which is now negative in the direction of the normal vector. Correspondingly, the opposite processes occurs when $v_n<0$ as precipitation takes place.

The volume changes coming from $K_m\neq 0$ have to be accounted for when choosing external boundary conditions for the flow, as e.g.~a closed container would lead to inconsistencies if there is overall net precipitation/dissolution inside the container. Instead the fluid must be allowed to escape/enter the domain in accordance with the mineral precipitation/dissolution. Hence, the necessary condition is
\begin{equation}
\int_{\partial \Omega_f(t)}m_f\mathbf v\cdot\mathbf n ds = -\int_{\Omega_f(t)}\partial_tm_fdV - \int_{\Gamma(t)}m_fK_mv_nds, \label{eq:sharpmassbc}
\end{equation}
where the integral on the left-hand side is over the external boundary of the fluid domain (that is, not facing a mineral), while the second integral on the right-hand side is over all fluid-mineral interfaces found inside the domain. Since we have constant density ($\partial_tm_f=0$), the inflow/outflow across the exterior boundary has to match the volume changes at the mineral boundary. If there is no expansion due to the reactions ($K_m=0$), (\ref{eq:sharpmassbc}) reduces to a standard boundary condition for incompressible flow.

\subsection{Domain evolution}
We mention for completeness how to model the evolving interface $\Gamma(t)$ in the case of a sharp-interface model. The location of the evolving interface $\Gamma(t)$, and hence also the evolution of the two domains, can be determined through a level-set approach. Letting $S:[0,\infty)\times\Omega\to\mathbb R$, where $\Omega=\Omega_f(t)\cup\Omega_m(t)\cup\Gamma(t)$, the level-set fulfills
\begin{equation*}
S(t,\mathbf x) = \left\{ \begin{array}{lr} <0 & \text{ if } \mathbf x\in\Omega_f(t), \\ 0 & \text{ if } \mathbf x\in\Gamma(t), \\ >0 & \text{ if } \mathbf x\in\Omega_m(t).
\end{array}
\right.
\end{equation*}
The level-set evolves following the level-set equation
\begin{equation*}
\partial_tS+v_n|\nabla S| = 0 \quad \text{ for } \mathbf x\in \Omega.
\end{equation*}
Note that the level-set equation requires the normal velocity of the evolving interface to be defined in the entire domain $\Omega$. This can be solved by, for each point $\mathbf x\in\Omega$, picking the corresponding value of $v_n$ from the point of $\Gamma(t)$ that is closest to $\mathbf x$.

\begin{remark}
	We will, without loss of generality, in the remainder of this paper only consider a simplified chemical setting, namely the case where $C_1+C_2 \leftrightharpoons C$; hence, one molecule of each type of solute is needed to form one mineral molecule. We will also assume that initial and boundary conditions for the two solutes are the same. In this case, the two solutes will always have the same concentration, which we will denote $c$. Hence only one equation (\ref{eq:sharpion}) is needed for the solute, and only one boundary condition (\ref{eq:sharpbcsol}). The reaction rates (\ref{eq:fpfd}) are modified such that $c_1=c_2$ and $n_1=n_2=1$.
\end{remark}

\section{Diffuse-interface models}\label{sec:diffuse}
Instead of dealing with a sharp interface, we now consider the transition between fluid and mineral to be a smooth transition. In this case there are no discontinuities between the two domains and all state variables will vary smoothly between the domains. To describe this setting we use a phase field $\phi(t,\mathbf{x})$. We will consider a phase field approaching the value 1 in the fluid domain and 0 in the mineral domain, and with a smooth transition of width $O(\lambda)$ between the domains. Hence, the interface between the fluid and mineral can no longer be said to be at an exact location, although it is common to let the value $\phi=0.5$ represent the interface location for practical purposes, which we will also apply in Section \ref{sec:limit}. The parameter $\lambda$ quantifies the width of the diffuse interface, and the sharp-interface model is expected to be recovered as $\lambda\to 0$, which we will investigate in Section \ref{sec:limit}. 

\subsection{Phase-field equations}
An Allen-Cahn equation \cite{allen1979cahn} is adopted for the phase field. The standard Allen-Cahn equation is not conservative and contain curvature effects that may be non-physical. In this section we will consider a standard Allen-Cahn equation as well as a reformulation to ensure conservation. 

The standard Allen-Cahn equation for mineral precipitation and dissolution is 
\begin{align}
\partial_t\phi  = - \frac{1}{\lambda^2}\gamma P'(\phi) + \gamma\nabla^2\phi-f_{\text{diff}}(\phi,T_f,c), \label{eq:AC}
\end{align}
Here, $P(\phi)=8\phi(1-\phi^2)$ is the double-well potential, ensuring that the phase field attains values close to 0 and 1 for small $\lambda$. The parameter $\gamma$ is a interface diffusion parameter and its role is explained in detail in Section \ref{sec:limit}.

The reaction rate $f_{\text{diff}}$ is related to the reaction rate for the sharp-interface model (\ref{eq:fsharp}), but is reformulated to incorporate the diffuse interface. We let
\begin{equation*}
f_{\text{diff}}(\phi,T,c) = \frac{4}{\lambda}\phi(1-\phi)\frac{1}{m_m}\big(f_p(T,c)-f_d(T,c)\big).
\end{equation*}
The factors $\phi(1-\phi)$ ensures that the reaction is only non-zero in the diffuse transition zone where $\phi$ is neither 0 nor 1. The fraction $4/\lambda$ ensures correct scaling as will be shown in section \ref{sec:limit}. For this reaction rate the Heaviside graph is not needed as $\phi\equiv1$ if all mineral is dissolved and the resulting reaction rate is hence 0 in this case. Hence, the precipitation $f_p$ and dissolution $f_d$ rates will be same as in (\ref{eq:fpfd}), but where $w\equiv1$ making the distance-measure superfluous. For convenience we will denote the net precipitation rate for $f(T,c) = f_p(T,c)-f_d(T,c)$ in the following.

As we will see in Section \ref{sec:conservation}, the Allen-Cahn equation is not conserving the phase-field variable. Inspired by \cite{chen2010cons,rubinstein1992nonlocal} we consider the following reformulation of the Allen-Cahn equation to overcome this:
	\begin{align}
	\partial_t\phi  = - \frac{1}{\lambda^2}\gamma P'(\phi) + \gamma\nabla^2\phi-f_{\text{diff}}(\phi,T,c) + \frac{\gamma}{\lambda^2}\frac{1}{|\Omega|}\int_{\Omega} P'(\phi)dV.
	\label{eq:rephase} 
	\end{align}
where $|\Omega|=\int_{\Omega}1dV$ is the measure of the domain $\Omega$. The reformulation (\ref{eq:rephase}) is conservative, which will be shown in Section \ref{sec:conservation}. We denote (\ref{eq:AC}) as the \emph{original} Allen-Cahn equation, and the reformulated (\ref{eq:rephase}) as the \emph{conservative} Allen-Cahn equations, respectively. The conservative reformulation (\ref{eq:rephase}) was recently also considered in \cite{bringedal2019FVCA}.

\subsection{Modified transport equations}
We now combine the conservation and transport equations from the fluid (\ref{eq:sharpfluid}), mineral (\ref{eq:sharpmineral}) and their separating boundary (\ref{eq:sharpbc}) to formulate a one-field model through incorporating the phase field. Letting $\phi$ be modeled by an Allen-Cahn equation, such that it approaches the value 1 in the fluid and 0 in the mineral, the following model equations describes the fluid flow and ensures conservation of mass, ions and energy in the combined fluid-mineral domain $\Omega$:

\begin{subequations}\label{eq:phase}
	\begin{align}
	\partial_t m_{\phi}+\nabla\cdot(m_f\phi\mathbf v) &= 0 \label{eq:phasemass} \\
	\partial_t(\rho_f\phi\mathbf v) + \nabla\cdot(\rho_f\phi\mathbf v\otimes\mathbf v) &= -\phi\nabla p + \nabla\cdot\boldsymbol{\sigma}_\phi - g(\phi,\lambda)\mathbf v + \frac{1}{2}\rho_f\mathbf v\partial_t\phi + \rho_f\mathbf g \label{eq:phasemom} \\
	\partial_t(\phi c+(1-\phi)m_m) +\nabla\cdot(c\phi\mathbf v) &= D\nabla(\phi\nabla c) \label{eq:phasesol} \\
	\partial_t((c_p\rho)_\phi T) + \nabla\cdot(c_{p,f}\rho_f T\phi\mathbf v) &= \nabla\cdot(k_\phi\nabla T) \label{eq:phaseT}
	\end{align}
\end{subequations}
Here, the newly introduced parameters $m_\phi$, $(c_p\rho)_\phi$ and $k_\phi$ are given by

\begin{equation*}
m_\phi = \phi m_f + (1-\phi) 2m_m, \quad (c_p\rho)_\phi = \phi c_{p,f}\rho_f+(1-\phi)c_{p,m}\rho_m, \quad k_\phi = \phi k_f + (1-\phi)k_m.
\end{equation*}
This way they attain the expected value in the fluid when $\phi=1$ and in the mineral when $\phi=0$. The factor two in front of $m_m$ is selected to ensure conservation of mass as $n_1=n_2=1$. The temperature $T$ describes the temperature of the entire system and hence replaces $T_f$ and $T_g$. The modified stress tensor is $\boldsymbol{\sigma}_\phi = \nu(\nabla\cdot(\phi\mathbf v))\mathbf I + \mu(\nabla(\phi\mathbf v)+(\nabla(\phi\mathbf v))^T)$. Hence, as the mixture of fluid and mineral is quasi-compressible due to the density differences, we here have a modified stress tensor, where $\nu$ is a bulk viscosity for the fluid-mineral mixture. 
The term $g(\phi,\lambda)\mathbf v$ is added to ensure $\mathbf v =0$ in the mineral \cite{bringedal2019phase}. The interpolation function $g(\phi,\lambda)$ is a decreasing, surjective and twice differentiable function fulfilling $g(1,\lambda)=0$ and $g(0,\lambda)>0$. This way, (\ref{eq:phasemom}) reduces to $\mathbf v=\mathbf 0$ when $\phi=0$. We use $g(\phi,\lambda)=\frac{K_{\mathbf v}}{\lambda}(1-\phi)$ for a constant $K_{\mathbf v}$, but mention that other possibilities are discussed in \cite{beckermann1999melt,bringedal2019phase,garcke2015shape}.

Note that, as pointed out already in Section \ref{sec:massbc}, we need to take care of which boundary conditions to use on $\partial\Omega$ for the mass conservation equation (\ref{eq:phasemass}). Directly letting $\phi\mathbf v\cdot\mathbf n=0$ on $\partial\Omega$ could lead to unphysical effects as the mass concentration $m_\phi$ can increase and decrease due to the chemical reaction (through changes in $\phi$). Hence, to let fluid escape or enter through $\partial\Omega$ accordingly, we need
\begin{equation}
\int_{\partial\Omega} m_{f}\phi\mathbf v\cdot\mathbf nds = -\int_{\Omega}\partial_tm_\phi dV.\label{eq:phasemassbc}
\end{equation} 
This condition is analog to (\ref{eq:sharpmassbc}). Note that if $m_f=2m_m$, then is $m_\phi=m_f$ and hence constant, which leads to the right-hand side of (\ref{eq:phasemassbc}) to be zero. This corresponds to the case $K_m=0$ as discussed in Section \ref{sec:massbc}.

\subsection{Conserved quantities of the phase-field models}\label{sec:conservation}
We here shortly show why the phase field described by the standard Allen-Cahn equation (\ref{eq:AC}) corresponds to a non-conserved order parameter. If we assume zero Neumann-conditions for the phase field on the external boundary $\partial\Omega$, we have
\begin{align*}
\frac{\text{d}}{\text{d} t} \int_{\Omega}\phi dV &= \int_{\Omega}\partial_t\phi dV = \int_{\Omega}\big(-\frac{1}{\lambda^2}\gamma P'(\phi)+\gamma\nabla^2\phi - \frac{4}{\lambda}\phi(1-\phi)\frac{1}{m_m}f(T,c)\big)dV \\
&= \int_{\Omega}\big(-\frac{1}{\lambda^2}\gamma P'(\phi)-\frac{4}{\lambda}\phi(1-\phi)\frac{1}{m_m}f(T,c)\big)dV + \int_{\partial\Omega}\gamma\nabla\phi\cdot\mathbf nds \\ &=\int_{\Omega}\big(-\frac{1}{\lambda^2}\gamma P'(\phi)-\frac{4}{\lambda}\phi(1-\phi)\frac{1}{m_m}f(T,c)\big)dV,
\end{align*}
which is non-zero also in the case of no chemical reactions. For the conservative Allen-Cahn (\ref{eq:rephase}), the phase field is conserved globally up to the chemical reaction, that is
\begin{align*}
\frac{\text{d}}{\text{d} t} \int_{\Omega}\phi dV &=\int_{\Omega}\big(-f_{\text{diff}}(\phi,T,c)\big)dV,
\end{align*}
hence we have $\frac{\text{d}}{\text{d} t}\int_{\Omega}\phi d\mathbf x = 0$ in the case of no chemical reactions. Note that since the integrated phase field gives the volume of the fluid part of the domain, this integral can be interpreted as the porosity of the system. This means that for the conservative Allen-Cahn equation (\ref{eq:rephase}) only the chemical reaction can affect the porosity. The scaling factor of $f_\text{diff}$, $\frac{4}{\lambda}\phi(1-\phi)$,  corresponds to the surface area of the fluid-mineral interface when integrated.

For total mass $m_\phi$, total amount of ions $\phi c+(1-\phi)m_m$ and total energy $(c_p\rho)_\phi T$ it is straightforward to show that these quantities are conserved under zero inflow ($\phi\mathbf v\cdot\mathbf n=0$) and zero Neumann ($\phi\nabla c\cdot\mathbf n=0$ and $k_\phi\nabla T\cdot\mathbf n=0$) boundary conditions on $\partial\Omega$, also in the case of chemical reactions. This is independent of which phase field is used to model the diffuse transition zone. If investigating conservation of $\phi c$, which corresponds to the amount of solute in the fluid, we get

\begin{align*}
\frac{\text{d}}{\text{d} t} \int_{\Omega}\phi c dV &= \int_{\Omega}\partial_t(\phi c )dV = m_m\int_{\Omega} \partial_t\phi dV,
\end{align*}
where the last integral is either connected to the reaction rate only, or to the reaction rate together with curvature effects (if using the original Allen-Cahn equation (\ref{eq:AC})). Hence, in the case of no chemical reaction, $\phi c$ is only conserved if using the conservative Allen-Cahn equation (\ref{eq:rephase}). In the case of chemical reactions, the change in $\phi c$ corresponds to the changes caused by the reaction rate.

\section{Sharp-interface limits of the diffuse-interface models}\label{sec:limit}
To justify that the considered phase-field model formulated in Section \ref{sec:diffuse} is meaningful, we consider the limit as the width of the diffuse interface $\lambda$ approaches zero to show that the model reduces to the expected sharp-interface model from Section \ref{sec:sharp}. 

We follow the ideas of \cite{caginalp1988dynamics} and separate between behavior far away from the interface and near it, and investigate the behavior as $\lambda\to 0$. We apply so-called outer expansions for the behavior of the phase-field model far away from the interface, and it will be shown that these outer expansions will reduce to the model equations in the fluid and mineral domain as $\lambda\to 0$. Behavior near the interface is described through so-called inner expansions of the unknowns and will be shown to capture the boundary conditions as $\lambda\to 0$. The two expansions are connected through matching conditions in a transition region where both are valid. We first present the two expansions and their matching conditions, before considering the outer and inner expansions of the phase-field models.

\subsection{Inner and outer expansions and their matching conditions}
Away from the interface, we consider the outer expansions of $\phi$, $c$, $p$, $\mathbf v$ and $T$. For $\phi$ this reads
\begin{equation}
\phi(t,\mathbf x) = \phi_0^{\text{out}}(t,\mathbf x) + \lambda\phi_1^{\text{out}}(t,\mathbf x) + \lambda^2\phi_2^{\text{out}}(t,\mathbf x) + \dots \label{eq:outer}
\end{equation}
and similarly for the other variables. Note that $\mu_f$ is considered (given) function of $T$ and hence do not need an expansion.

The inner expansions are valid near the interface, and it is hence convenient to make a change of coordinates into local, curvilinear coordinates. We let $\Gamma_\lambda(t)$ denote the set of points $\mathbf y_{\lambda}\in\Omega$ such that $\phi(t,\mathbf y_\lambda) = 1/2$. Note that the location of these points $\mathbf y_\lambda$ vary with time as the phase field evolves. 
Let $\mathbf s$ be an arc-length parametrization of the points on $\Gamma_\lambda(t)$ (note that $\mathbf s$ is a scalar if $\Omega\subset \mathbb R^2$ and consists of two coordinates if $\Omega\subset \mathbb R^3$) and $\mathbf n_\lambda$ the normal vector at $\Gamma_\lambda(t)$ pointing into the mineral (that is, towards decreasing $\phi$). Further, we let $r$ be the signed distance from a point $\mathbf x$ to the interface, being negative when inside the fluid ($\phi>0.5$) and positive inside the mineral ($\phi<0.5$). This way we can let each point $\mathbf x$ be described through
\begin{equation}
\mathbf x = \mathbf y_{\lambda}(t,\mathbf s) + r\mathbf n_\lambda(t,\mathbf s).\label{eq:coord}
\end{equation}
Through this transform, $(r,\mathbf s)$ makes up a new, orthogonal coordinate system. It is shown in \cite{caginalp1988dynamics} that 
\begin{equation*}
|\nabla r|=1,\ \ \ \nabla r\cdot\nabla s_i = 0, \ \ \ \partial_t r = -v_n, \ \ \ \nabla^2 r = \frac{\kappa+2\Pi r}{1+\kappa r+\Pi r^2},
\end{equation*}
where $\kappa$ and $\Pi$ are the mean and Gaussian curvature of the interface. We assume that $\mathbf y_\lambda$ and $\mathbf n_\lambda$ fulfill expansions $\mathbf y_\lambda = \mathbf y_0 + O(\lambda)$ and $\mathbf n_\lambda = \mathbf n_0 + O(\lambda)$ where $\mathbf y_0$ is a point of the interface $\Gamma_0^{\text{out}}(t)$ defined through $\phi_0^{\text{out}}=1/2$ and $\mathbf n_0$ being the normal vector of $\Gamma_0^{\text{out}}(t)$ pointing into the mineral. Then, with $z=r/\lambda$ we consider the inner expansions
\begin{equation}
\phi(t,\mathbf x) = \phi_0^{\text{in}}(t,z,\mathbf s) + \lambda\phi_1^{\text{in}}(t,z,\mathbf s) + \lambda^2\phi_2^{\text{in}}(t,z,\mathbf s) + \dots, \label{eq:inner}
\end{equation}
and similarly for the other unknowns. In the local, curvilinear coordinates and when accounting for the scaling of $z$, we need to rewrite derivatives. For a generic variable $u$ or $\mathbf u$, we obtain \cite{caginalp1988dynamics}
\begin{subequations}\label{eq:derivatives}
\begin{align}
\partial_t u &= -\lambda^{-1}v_{n,0}\partial_zu^{\text{in}} + (\partial_t+\partial_t s\cdot\nabla_{\mathbf s})u^{\text{in}} + O(\lambda), \\
\nabla_x u &= \lambda^{-1}\partial_z u^{\text{in}}\mathbf n_0 + \nabla_{\mathbf s}u^{\text{in}} + O(\lambda), \\
\nabla_x\cdot \mathbf u &= \lambda^{-1}\partial_z\mathbf u^{\text{in}}\cdot \mathbf n_0 + \nabla_{\mathbf s}\cdot\mathbf u^{\text{in}} + O(\lambda),\\
\nabla_x^2 v &= \lambda^{-2}\partial_{z}^2u + \lambda^{-1}\kappa_0\partial_zu + O(1), \label{eq:laplacemathching}
\end{align}
\end{subequations}
where we have used that $\nabla_x^2r=\kappa_0+O(\lambda)$ is the lowest order mean curvature and $v_n=v_{n,0}+O(\lambda)$. Here, $\kappa_0$ and $v_{n,0}$ are the curvature and normal velocity of the interface $\Gamma_0^{\text{out}}(t)$. In the last condition (\ref{eq:laplacemathching}) we have used the properties $|\nabla r|=1$ and $\nabla r\cdot \nabla s_i=0$.

The outer (\ref{eq:outer}) and inner (\ref{eq:inner}) expansions are tied together via (\ref{eq:coord}). For fixed $t$ and $\mathbf s$, we let $\mathbf y_{1/2\pm}$ denote the limit $r\to 0^+$ (i.e., from the mineral side) and $r\to 0^-$ (i.e., from the fluid side) of $\mathbf x$. 
We demand that when the outer expansions approach the boundary they need to match the inner expansions as they leave the boundary. That is, the limit when $r\to 0^{\pm}$ of the outer expansions need to match the limit when $z\to\pm\infty$ of the inner expansions. Using Taylor expansions, one can show that the following matching conditions
\begin{subequations}\label{eq:matching}
	\begin{align}
	\lim_{z\to\pm\infty}\phi_0^{\text{in}}(t,z,\mathbf s) &= \phi_0^{\text{out}}(t,\mathbf y_{1/2\pm}),\label{eq:match1} \\
	\lim_{z\to\pm\infty}\partial_z\phi_0^{\text{in}}(t,z,\mathbf s) &= 0, \label{eq:match2}\\
	\lim_{z\to\pm\infty}\partial_z\phi_1^{\text{in}}(t,z,\mathbf s) &= \nabla \phi_0^{\text{out}}(t,\mathbf y_{1/2\pm})\cdot\mathbf n_0, \label{eq:mathc3}
	\end{align}
\end{subequations}
are fulfilled, and similarly for the other variables \cite{caginalp1988dynamics}.

\subsection{Outer expansions}
We first apply the outer expansions (\ref{eq:outer}) to the phase-field equations (\ref{eq:AC}) and (\ref{eq:rephase}). Inserting the outer expansions in (\ref{eq:AC}), (\ref{eq:rephase}), the dominating term when $\lambda\to 0$ will in both cases be
\begin{equation*}
P'(\phi_0^{\text{out}}) = 0.
\end{equation*} 
Recalling that $P(\phi) = 8\phi^2(1-\phi)^2$, the above equation has the three solutions $\phi_0^{\text{out}}=0$, $\phi_0^{\text{out}}=0.5$ and $\phi_0^{\text{out}}=1$. Only $\phi_0^{\text{out}}=0$ and $\phi_0^{\text{out}}=1$ give stable solutions. We let $\Omega_{m,0}(t)$ and $\Omega_{f,0}(t)$ be the collection of points $\mathbf x\in\Omega$ associated with where $\phi_0^{\text{out}}=0$ and $\phi_0^{\text{out}}=1$, respectively.

For the remaining phase-field model equations (\ref{eq:phase}), we consider the two limit cases of $\phi_0^{\text{out}}=0$ and $\phi_0^{\text{out}}=1$ separately. When $\phi_0^{\text{out}}=0$ (that is, in $\Omega_{m,0}(t)$), the dominating terms of the outer expansions (\ref{eq:outer}) in (\ref{eq:phase}) give that
\begin{subequations}\label{eq:phasesharpmineral}
	\begin{align}
	\partial_tm_m &= 0, \\
	\mathbf v_0^{\text{out}} &=\mathbf{0}, \\
	\partial_tm_m &= 0, \\
	\partial_t(c_{p,m}\rho_m T_0^{\text{out}}) &= k_m\nabla^2 T_0^{\text{out}}.
	\end{align}
\end{subequations}
The first and third equations simply tell us that the mineral molar density is constant, which was also a model assumption, cf.~Section \ref{sec:sharpmineral}. The equations (\ref{eq:phasesharpmineral}) correspond to the ones formulated in (\ref{eq:sharpmineral}) for the mineral domain.
When $\phi_0^{\text{out}}=1$ (that is, in $\Omega_{f,0}(t)$), the dominating terms of the outer expansions (\ref{eq:outer}) in (\ref{eq:phase}) give that
\begin{subequations}\label{eq:phasesharpfluid}
	\begin{align}
	\partial_tm_f+\nabla\cdot(m_f\mathbf v_0^{\text{out}}) &= 0, \label{eq:phasesharpmass}\\
	\partial_t(\rho_f\mathbf v_0^{\text{out}}) + \nabla\cdot(\rho_f\mathbf v_0^{\text{out}}\otimes\mathbf v_0^{\text{out}}) &= -\nabla p_0^{\text{out}} + \nabla\cdot\boldsymbol{\sigma}_0 + \rho_f\mathbf g, \\
	\partial_t c_0^{\text{out}} + \nabla\cdot(c_0^{\text{out}}\mathbf v_0^{\text{out}}) &= D\nabla^2c_0^{\text{out}}, \\
	\partial_t(c_{p,f}\rho_fT_0^{\text{out}}) + \nabla\cdot(c_{p,f}\rho_fT_0^{\text{out}}) &= k_f\nabla^2T_0^{\text{out}},
	\end{align}
\end{subequations}
where $\boldsymbol{\sigma}_0 = \nu(\nabla\cdot\mathbf v_0^{\text{out}})\mathbf I + \mu(T_0^\text{out})(\nabla\mathbf v_0^{\text{out}} + (\nabla\mathbf v_0^{\text{out}})^T)$. Since $m_f$ is constant, (\ref{eq:phasesharpmass}) reduces to $\nabla\cdot\mathbf v_0^{\text{out}}=0$, which also means that we can simplify the stress tensor to $\boldsymbol{\sigma}_0 =  \mu(T_0^\text{out})\nabla\mathbf v_0^{\text{out}}$. Note that we have here assumed that $\mu(T)$ is Lipschitz-continuous. 
The equations (\ref{eq:phasesharpfluid}) hence correspond to the ones formulated in (\ref{eq:sharpfluid}) for the fluid domain.

\subsection{Inner expansions}
We now apply the inner expansions (\ref{eq:inner}) to the model equations (\ref{eq:AC}), (\ref{eq:rephase}) and (\ref{eq:phase}) to recover the boundary conditions at the evolving interface.

\subsubsection{Phase-field equations}\label{sec:phaselimit}
For the original Allen-Cahn equation (\ref{eq:AC}), the dominating $O(\lambda^{-2})$ terms when inserting the inner expansions (\ref{eq:inner}) (and keeping in mind the rescaled derivatives (\ref{eq:derivatives})) are
\begin{equation}
P'(\phi_0^{\text{in}}) =\partial_{z}^2\phi_0^{\text{in}}.\label{eq:phasein0}
\end{equation}
From the first matching condition (\ref{eq:match1}) we have that $\lim_{z\to\infty}\phi_0^{\text{in}} = 0$ and $\lim_{z\to\infty}\phi_0^{\text{in}} = 1$. By multiplying (\ref{eq:phasein0}) by $\partial_z\phi_0^{\text{in}}$, integrating twice with respect to $z$ and using the matching conditions and that $\phi_0^{\text{in}}(z=0) = 0.5$, we get that
\begin{equation}
\phi_0^{\text{in}}(t,z,\mathbf s) = \phi_0^{\text{in}}(z) = \frac{1}{1+e^{4z}} = \frac{1}{2}\tanh(2z) \label{eq:innerphi}
\end{equation}
fulfills the equation. For the $O(\lambda^{-1})$ terms one obtains 
\begin{equation*}
\gamma\big(P''(\phi_0^{\text{in}})-\partial_z^2\big)\phi_1^{\text{in}} = ( v_{n,0}+\gamma\kappa_0)\partial_z\phi_0^{\text{in}} - 4 \phi_0^{\text{in}}(1-\phi_0^{\text{in}})\frac{1}{m_m}f(T_0^{\text{in}},c_0^{\text{in}}).
\end{equation*}
We follow the same steps as in \cite{redeker2016phase} and view the left-hand side as an operator $\mathcal F$ depending on $\phi_0^{\text{in}}$ and applied to $\phi_1^{\text{in}}$. 
The operator $\mathcal F$ is a Fredholm operator of index zero, hence the above equation has a solution if and only if the right-hand side, denoted by  $A(\phi_0^{\text{in}})$, is orthogonal to the kernel of $\mathcal F$. It follows from \eqref{eq:phasein0} that $\partial_z\phi_0^{\text{in}}$ is in the kernel of $\mathcal F$. Since $v_{n,0}$, $\kappa_0$, $T_0^{\text{in}}$ and $c_0^{\text{in}}$ are independent of $z$ (the two latter will be shown in the following subsections), the solvability condition implies
\begin{align*}
0 &=\int_{-\infty}^{\infty}A(\phi_0^{\text{in}})\partial_z\phi_0^{\text{in}}dz \\
&= (v_{n,0} + \gamma\kappa_0)\int_{-\infty}^{\infty}(\partial_z\phi_0^{\text{in}})^2dz - 4  \frac{1}{m_m}f(T_0^{\text{in}},c_0^{\text{in}})\int_{-\infty}^{\infty}\phi_0^{\text{in}}(1-\phi_0^{\text{in}})\partial_z\phi_0^{\text{in}}dz \\
&= \frac{2}{3}(v_{n,0}+\gamma\kappa_0+ \frac{1}{m_m}f(T_0^{\text{in}},c_0^{\text{in}})).
\end{align*}
We apply the matching condition (\ref{eq:match1}) for $T$ and $c$ at the moving interface and obtain
\begin{equation*}
v_{n,0} = -\gamma\kappa_0 - \frac{1}{m_m}f(T_0^{\text{out}}(t,\mathbf y_{1/2-}),c_0^{\text{out}}(t,\mathbf y_{1/2-})),
\end{equation*}
which corresponds to the boundary condition \eqref{eq:fsharp} at the moving interface, with an additional curvature-driven motion, which is responsible for the Allen-Cahn equation not being conservative. Since the Allen-Cahn equation is derived from curvature-driven flow, it is expected that this term appears \cite{allen1979cahn}. Although the curvature effects would be zero if $\gamma=0$, this is not a valid choice for the Allen-Cahn equation (\ref{eq:AC}).

We address the sharp-interface limit of the conservative Allen-Cahn equation \eqref{eq:rephase}. We follow similar steps as \cite{chen2010cons} who investigated the sharp-interface limit for a conservative Allen-Cahn equation without chemical reactions. The equation \eqref{eq:rephase} is first split in two equations:
\begin{subequations}
	\begin{eqnarray}
	\partial_t\phi  = -\frac{1}{\lambda^2}\gamma P'(\phi) + \gamma\nabla^2\phi -f_{\text{diff}}(\phi,T,c) + \frac{1}{\lambda}\gamma\omega(t) \label{eq:rephase1}\\
	\omega(t) = \frac{1}{\lambda}\frac{1}{|\Omega|}\int_{\Omega}P'(\phi)dV. \label{eq:rexi}
	\end{eqnarray}
\end{subequations}
The average of $P'(\phi)$ appearing in \eqref{eq:rexi} is small and will decrease as $\lambda$ decreases. It is shown in \cite{chen2010cons} that $\omega(t) = O(\lambda^0)$ as $\lambda\to0$. 
The inner expansions (\ref{eq:inner}) of \eqref{eq:rephase1} are dominated by the terms
\begin{equation*}
P'(\phi_0^{\text{in}}) =  \partial_z^2 \phi_0^{\text{in}},
\end{equation*}
which has solution (\ref{eq:innerphi}). 
The next order terms are affected by the newly introduced term:
\begin{equation*}
\big(P''(\phi_0^{\text{in}})-\partial_z^2\big)\phi_1^{\text{in}} = (  v_{n,0}+\gamma\kappa_0)\partial_z\phi_0^{\text{in}} - 4 \phi_0^{\text{in}}(1-\phi_0^{\text{in}})\frac{1}{m_m}f(T_0^{\text{in}},c_0^{\text{in}}) + \gamma\omega_0^{\text{in}},
\end{equation*}
where $\omega_0^{\text{in}} =\frac{1}{\lambda}\frac{1}{|\Omega|}\int_{\Omega}P'(\phi_0^{\text{in}})dV$. 
As with the original Allen-Cahn equation we can use a solvability argument to proceed. The above equation only has a solution when
\begin{equation*}
\int_{-\infty}^{\infty} \Big((  v_{n,0}+\gamma\kappa_0)\partial_z\phi_0^{\text{in}} - 4 \phi_0^{\text{in}}(1-\phi_0^{\text{in}})\frac{1}{m_m}f(T_0^{\text{in}},c_0^{\text{in}}) + \gamma\omega_0^{\text{in}}\Big)\partial_z\phi_0^{\text{in}}dz=0.
\end{equation*}
Following same argumentation as with the original Allen-Cahn equation, and since $\omega_0^{\text{in}}$ is independent of $z$ by definition, we get
\begin{equation*}
v_{n,0} = -\gamma \kappa_0 - \frac{1}{m_m}f(T_0^{\text{in}},c_0^{\text{in}}) + 3\gamma\omega_0^{\text{in}}.
\end{equation*}
Finally, \cite{chen2010cons} shows that the condition of $\int_{\Omega}\partial_t\phi_0^{\text{in}}d\mathbf x=0$ is equivalent to
\begin{equation*}
\omega_0^{\text{in}} = \frac{1}{3}\frac{1}{|\Gamma(t)|}\int_{\Gamma(t)}\kappa_0ds=\frac{1}{3}\overline{\kappa}_0,
\end{equation*}
where $\overline{\kappa}$ is the average curvature along $\Gamma(t)$.
We hence arrive at
\begin{equation*}
v_{n,0} = -\gamma (\kappa_0-\overline{\kappa}_0) - \frac{1}{m_m}f(T_{0}^{\text{out}}(t,\mathbf y_{1/2-}),c_0^{\text{out}}(t,\mathbf y_{1/2-})) \quad \text{on } \Gamma(t),
\end{equation*}
This means that the interface velocity is still driven by both the chemical reaction and by curvature, but where the curvature-driven movement fulfills conservation of the phase-field parameter. The curvatured-driven motion redistributes the mineral towards constant curvature; towards bubbles \cite{schlogl1972chemical}. Hence, the fluid-solid interface moves due to the chemical reaction, and additionally in a conservative manner by redistributing mineral towards constant curvature.

\subsubsection{Mass conservation equation}
Inserting the inner expansions (\ref{eq:inner}) into the mass conservation equation (\ref{eq:phasemass}) gives the dominating $O(\lambda^{-1})$ term
\begin{equation*}
-v_{n,0}\partial_z m_\phi + \partial_z(m_f\phi_0^{\text{in}}\mathbf v_0^{\text{in}})\cdot\mathbf n_0 = 0.
\end{equation*}
Integrating with respect to $z$ from $-\infty$ to $+\infty$ and using matching conditions (\ref{eq:match1}) give
\begin{equation}
-m_f\mathbf v_0^{\text{in}}(t,\mathbf y_{1/2-})\cdot\mathbf n_0 = v_{n,0}(2m_m-m_f).\label{eq:phaselimmass}
\end{equation}
This boundary condition corresponds to (\ref{eq:sharpbcmass}).

\subsubsection{Momentum conservation equation}
The dominating $O(\lambda^{-2})$ terms arising from inserting the inner expansions (\ref{eq:inner}) into (\ref{eq:phasemom}) are 
\begin{equation*}
\mu\partial_{z}^2(\phi_0^{\text{in}}\mathbf v_0^{\text{in}}) + (\nu+\mu)\partial_{z}^2(\phi_0^{\text{in}}\mathbf v_0^{\text{in}}\cdot\mathbf n_0)\mathbf n_0 = 0
\end{equation*}
Note that $\mu$ and $\nu$ can be functions of $T_0^{\text{in}}$, which is still independent of $z$. 
We integrate twice with respect to $z$ and apply matching conditions (\ref{eq:match1}) and (\ref{eq:match2}) and hence arrive at
\begin{equation*}
\mu\mathbf v_0^{\text{in}}(t,\mathbf y_{1/2-}) = -(\nu+\mu)(\mathbf v_0^{\text{in}}(t,\mathbf y_{1/2-})\cdot\mathbf n_0)\mathbf n_0.
\end{equation*}
Here, the temperature $T_0^{\text{in}}$ inside $\mu$ and $\nu$ is also to be evaluated at $\mathbf y_{1/2-}$. 
Note that the current equation already gives us a no-slip condition: The velocity has only a normal component at the boundary, and hence no tangential component. 
Combining with (\ref{eq:phaselimmass}) we can write this condition as
\begin{equation*}
\mathbf v_0^{\text{in}}(t,\mathbf y_{1/2-}) = -\frac{\nu+\mu}{\mu}\frac{m_f-2m_m}{m_f}v_{n,0}\mathbf n_0.
\end{equation*}
However, the current condition is only consistent with (\ref{eq:phaselimmass}) when $\nu=-2\mu$. In this case,
\begin{equation*}
\mathbf v_0^{\text{in}}(t,\mathbf y_{1/2-}) = \frac{m_f-2m_m}{m_f}v_{n,0}\mathbf n_0,
\end{equation*}
which then corresponds to (\ref{eq:sharpbcmom}). Hence, this relation between the two viscosities $\nu$ and $\mu$ give consistency and the expected boundary condition in the sharp-interface limit. The second viscosity $\nu$ was here introduced since the fluid-mineral mixture is quasi-incompressible. A second viscosity is usually connected to how a fluid reacts on an fluid expansion/contraction and is here chosen in relation to $\mu$ to ensure that the chosen sharp-interface boundary condition is obtained.

\begin{remark}
	In case $m_f=2m_m$, we have no expansion/contraction due to the mineral precipitation and dissolution. In this case $m_\phi=m_f$ and (\ref{eq:phasemass}) would have expressed $\nabla\cdot(\phi\mathbf v)=0$ as $m_f$ is constant. Then, the second viscosity $\nu$ does not have to be introduced as (\ref{eq:phasemom}) could have used a simpler stress tensor. The sharp-interface limit in this case gives directly $\mathbf v_0^{\text{in}}(t,\mathbf y_{1/2-})=\mathbf 0$.
\end{remark}

\subsubsection{Solute conservation equation}
Inserting the inner expansions (\ref{eq:inner}) into the solute conservation equation (\ref{eq:phasesol}), the dominating $O(\lambda^{-2})$ term is
\begin{equation*}
\partial_z(\phi_0^{\text{in}}\partial_z c_0^{\text{in}}) = 0.
\end{equation*}
Integrating with respect to $z$ and using the matching condition (\ref{eq:match2}) together with the fact that $\phi_0^{\text{in}}>0$, we get that
\begin{equation*}
\partial_{z}c_0^{\text{in}}=0.
\end{equation*}
Hence, $c_0^{\text{in}}$ is independent of $z$. For the $O(\lambda^{-1})$ terms we obtain
\begin{equation*}
-v_{n,0}\partial_z\big(\phi_0^{\text{in}}c_0^{\text{in}}+(1-\phi_0^{\text{in}})m_m\big) + \partial_z(c_0^{\text{in}}\phi_0^{\text{in}}\mathbf v_0^{\text{in}})\cdot\mathbf n_0 = D\partial_z(\phi_0^{\text{in}}\partial_z c_1^{\text{in}}).
\end{equation*}
Integrating with respect to $z$ from $-\infty$ to $+\infty$ and using matching conditions (\ref{eq:matching}) we get
\begin{equation*}
v_{n,0}\big(m_m-c_0^{\text{out}}(t,\mathbf y_{1/2-})\big) = \big(-c_0^{\text{out}}(t,\mathbf y_{1/2-})\mathbf v_0^{\text{out}}(t,\mathbf y_{1/2-}) + D\nabla c_0^{\text{out}}(t,\mathbf y_{1/2-})\big)\cdot\mathbf n_0,
\end{equation*}
which corresponds to (\ref{eq:sharpbcsol}).

\subsubsection{Energy conservation equation}
By inserting the inner expansions (\ref{eq:inner}) into (\ref{eq:phaseT}), the dominating $O(\lambda^{-2})$ term is
\begin{equation*}
\partial_{z}^2 T_0^{\text{in}} = 0.
\end{equation*}
When integrating this equation with respect to $z$ and using the matching condition (\ref{eq:match2}), it becomes clear that
\begin{equation*}
\partial_z T_0^{\text{in}} = 0;
\end{equation*}
that is, $T_0^{\text{in}}$ is independent of $z$. The $O(\lambda^{-1})$ terms are
\begin{equation*}
-v_{n,0}\partial_z((c_{p}\rho)_{\phi^0}T_0^{\text{in}}) + \partial_z(c_f\rho_fT_0^{\text{in}}\phi_0^{\text{in}}\mathbf v_0^{\text{in}})\cdot\mathbf n_0 = \partial_z(k_{\phi^0}\partial_z T_1^{\text{in}}),
\end{equation*}
where the notation $(c_p\rho)_{\phi^0} = \phi_0^{\text{in}}c_{p,f}\rho_f + (1-\phi_0^{\text{in}})c_{p,m}\rho_m$ and similarly for $k_{\phi^0}$. We integrate with respect to $z$ from $-\infty$ to $+\infty$, taking advantage of that $T_0^{\text{in}}$ is independent of $z$. After applying matching conditions (\ref{eq:matching}) we get that
\begin{align*}
v_{n,0}&\big(c_{p,m}\rho_mT_0^{\text{out}}(t,\mathbf y_{1/2+})-c_{p,f}\rho_fT_0^{\text{out}}(t,\mathbf y_{1/2-})\big)= \big(-k_m\nabla T_0^{\text{out}}(t,\mathbf y_{1/2+})\\ &- c_{p,f}\rho_fT_0^{\text{out}}(t,\mathbf y_{1/2-})\mathbf v_0^{\text{out}}(t,\mathbf y_{1/2-}) + k_f\nabla T_0^{\text{out}}(t,\mathbf y_{1/2-})\big)\cdot\mathbf n_0,
\end{align*}
which corresponds to (\ref{eq:sharpbcT}) when associating the temperature at the positive (mineral) side with $T_m$ and the temperature at the negative (fluid) side with $T_f$.

\section{Numerical discretization of the phase-field model}\label{sec:numerics}
We here numerically discretize the resulting phase-field model. Special attention is given to the conservative Allen-Cahn equation to ensure that the numerical discretization honors the conservative property. We apply Finite Volume (FV) discretization in space, which is outlined in Section \ref{sec:FV}, and it is shown under which conditions the conservative Allen-Cahn equation (\ref{eq:rephase}) is discretely conservative in Section \ref{sec:discreteAC}. The conservative Allen-Cahn equation is non-linear and challenging to solve with Newton iterations due to the non-local term, and we formulate L-scheme iterations which converge to a conservative solution in Section \ref{sec:Lscheme}. The conservative Allen-Cahn equation is coupled to the other model equations in an iterative manner, which in a simplified setup is proven to converge in Section \ref{sec:coupmod}.

\subsection{Finite Volume discretization}\label{sec:FV}
For numerical discretization, we apply a cell-centered FV scheme on an admissible mesh $\mathcal E$ \cite{eymard2000finite}, where the phase-field variable (and similarly for temperature and solute) is approximated in the cell-center and we calculate fluxes across cell edges. The fluxes are for simplicity two-point approximations, but extensions using multi-point flux approximations are possible. By integrating the conservative Allen-Cahn equation (\ref{eq:rephase}) across one element and using Gauss' theorem on the diffusive term, the discrete FV approximation of (\ref{eq:rephase}) reads
\begin{align}
\frac{\phi_K^{n+1}-\phi_K^n}{\Delta t} + \frac{\gamma}{\lambda^2} P'(\phi_K^\ell) = &\frac{\gamma}{|K|}\sum_{L\in \mathcal N(K)} |\sigma_{K,L}| F^{\phi,\ell}_{K,L} - \frac{4}{\lambda} \phi_K^\ell(1-\phi_K^\ell)\frac{f(T_K^\ell,c_K^\ell)}{m_m}\nonumber\\ &+ \frac{\gamma}{\lambda^2}\frac{1}{|\Omega|}\sum_{J\in\mathcal E}|J|P'(\phi^\ell_J),\label{eq:discAC}
\end{align}
where $|K|$ is the measure of element $K$. We have here also discretized in time, and let superscript $n$ and $n+1$ denote the time step. The notation $\phi_K$ refers to the cell-centered value of $\phi$ in $K$ and is assumed to approximate the value of $\phi$ in entire $K$. Further, $\mathcal N(K)$ refers to the neighboring elements of $K$ and $|\sigma_{K,L}|$ is the measure of the edge $\sigma_{K,L}$ between element $K$ and the neighboring element $L$. The integral $\int_{\Omega}P'(\phi)d\mathbf x$ has been approximated by the sum $\sum_{J\in\mathcal E}|J|P'(\phi_J^\ell)$. The superscript $\ell$ denotes time step, and is $n$ or $n+1$ depending on whether forward or backward Euler is applied, respectively. The possible choices for time stepping will be addressed later. The fluxes $F_{K,L}^{\phi,\ell}$ approximate the diffusive flux $\nabla^2\phi$ on the edge between elements $K$ and $L$ and are given by
\begin{equation*}
F_{K,L}^{\phi,\ell} = \frac{\phi_L^\ell-\phi_K^\ell}{d_{K,L}},
\end{equation*}
where $d_{K,L}$ is the Euclidean distance between the center points $x_K\in K$ and $x_L\in L$. We have $F_{K,L}^{\phi,\ell} = -F_{L,K}^{\phi,\ell}$ on interior edges, which ensures the conservation between elements.

\subsection{Discrete conservation of the conservative Allen-Cahn equation}\label{sec:discreteAC}
We consider the global conservation of (\ref{eq:discAC})  up to the chemical reaction by summing up (\ref{eq:discAC}) for all $K\in\mathcal E$, hence across entire domain $\Omega$. Summing up over all $K\in\mathcal E$ and using Neumann boundary conditions for $\phi$ and that $F_{K,L}^{\phi,\ell} = -F_{L,K}^{\phi,\ell}$, we obtain
\begin{align*}
\sum_{K\in\mathcal E} \phi^{n+1}_K|K| = &\sum_{K\in\mathcal E} \phi^n_K|K| - \sum_{K\in\mathcal E}\frac{4}{\lambda} \phi_K^\ell(1-\phi_K^\ell)\frac{f(T_K^\ell,c_K^\ell)}{m_m}|K|\\ &- \frac{\gamma}{\lambda^2}\sum_{K\in\mathcal E}\big(|K|P'(\phi_K^\ell) - \frac{|K|}{|\Omega|}\sum_{J\in\mathcal E}|J|P'(\phi^\ell_J)\big)
\end{align*}
The summation $\sum_{K\in\mathcal E}\big(|K|P'(\phi_K^\ell) - \frac{|K|}{|\Omega|}\sum_{J\in\mathcal E}|J|P'(\phi^\ell_J)\big)$ adds up to zero under certain restrictions on the time evaluation. Denoting the sum $\sum_{J\in\mathcal E}|J|P'(\phi^\ell_J) := \Pi$, we have that the summation can be rewritten to
\begin{equation*}
\sum_{K\in\mathcal E}\big(|K|P'(\phi_K^\ell) - \frac{|K|}{|\Omega|}\sum_{J\in\mathcal E}|J|P'(\phi^\ell_J)\big) = \Pi - \sum_{K\in\mathcal E}\frac{|K|}{|\Omega|}\Pi.
\end{equation*}
As $|\Omega|$ and $\Pi$ are independent of the element, it is obvious that the terms sum up to zero as $\sum_{K\in\mathcal E}|K|=|\Omega|$ as the mesh is conforming. This remains true as long as the same choice of $\ell$ is taken for the evaluation of $P'(\phi^\ell_K)$ in both appearances for corresponding elements. This means that both forward Euler and backward Euler will be discretely conservative, as also shown in \cite{bringedal2019FVCA}. However, as also noted in \cite{bringedal2019FVCA}, an explicit discretization of the non-linear terms is unstable, while the implicit choice is slow to solve with Newton iterations as the Jacobian is full. In the following we suggest a strategy to overcome these difficulties, while still assuring the discrete conservation of (\ref{eq:rephase}).

\subsection{L-scheme iterations of a semi-implicit scheme}\label{sec:Lscheme}
We here consider a particular semi-implicit version of (\ref{eq:rephase}); and consider an alternative non-linear solver of the resulting non-linear system of equations. We will consider the L-scheme \cite{pop2004lscheme,list2016iterative}, which reformulates the discrete equation to iterations which form a contraction. To simplify the notation, we consider only the time-discrete equation. Initially we let the scheme be implicit, hence $\ell=n+1$, and we address how to solve
\begin{align}\label{eq:ACtime}
\frac{1}{\Delta t}(\phi^{n+1}-\phi^n) + \frac{\gamma}{\lambda^2}P'(\phi^{n+1}) = &\gamma \nabla^2\phi^{n+1} - \frac{4}{\lambda}\phi^{n+1}(1-\phi^{n+1})\frac{f(T,c)}{m_m}\nonumber\\ &+ \frac{\gamma}{\lambda ^2}\frac{1}{|\Omega|}\int_{\Omega}P'(\phi^{n+1})dV.
\end{align}
As we only address the non-linearity with respect to the phase field, we assume the temperature and concentration in the reactive term given. 
Note that the L-scheme approach is of course also applicable to solve the non-linearities in (\ref{eq:AC}), which has been done in \cite{BringedalBastidas3}, but we here focus on how to apply it to (\ref{eq:rephase}), and show how the L-scheme iterations can be assured to converge to the conservative solution.

\subsubsection{Convergence of the L-scheme}\label{sec:Lschemeconv}
The L-scheme adds an extra stabilization term to create a contraction, which converges by the Banach fixed point theorem. We need the non-linearity in the resulting discrete equation to be monotonically increasing or decreasing. However, the non-linearity in (\ref{eq:ACtime});
\begin{equation}
G(\phi,T,c) = -\frac{\gamma}{\lambda^2}P'(\phi) - \frac{4}{\lambda}\phi(1-\phi)\frac{f(T,c)}{m_m} + \frac{\gamma}{\lambda^2}\frac{1}{|\Omega|}\int_{\Omega}P'(\phi)dV,\label{eq:Fnonlin}
\end{equation}
is both increasing and decreasing with respect to $\phi$. Therefore we split the non-linearity $G(\phi,T,c)$ into an increasing and decreasing part, and solve the increasing part explicitly and decreasing implicitly. Note that we consider $T,c$ known and fixed in the following. Hence, we define
\begin{align*}
G_-(\phi,T,c) &= \int_0^\phi \min\{0,\partial_1G(\psi,T,c)\}d\psi, \\ G_+(\phi,T,c) &= \int_0^\phi \max\{0,\partial_1G(\psi,T,c)\}d\psi,
\end{align*}
and consider the time discretization
\begin{equation}\label{eq:ACtimepm}
\frac{1}{\Delta t}(\phi^{n+1}-\phi^n)  = \gamma \nabla^2\phi^{n+1} + G_-(\phi^{n+1},T,c) + G_+(\phi^n,T,c).
\end{equation}
The L-scheme iterations solve the non-linear equation (\ref{eq:ACtimepm}) iteratively through adding a stabilizing therm. By letting $j$ be the iteration index and setting $\phi^{n+1,0}:=\phi^n$, we solve
\begin{align}\label{eq:ACLscheme}
\frac{1}{\Delta t} (\phi^{n+1,j+1}-\phi^n) = &\gamma\nabla^2\phi^{n+1,j+1}+ G_-(\phi^{n+1,j},T,c)+G_+(\phi^n,T,c)\nonumber\\ &- \mathcal L(\phi^{n+1,j+1}-\phi^{n+1,j}),
\end{align}
where $\mathcal L\in\mathbb R_+$. 

\begin{theorem}\label{thm:Lscheme}
	The L-scheme iterations (\ref{eq:ACLscheme}) form a contraction and are guaranteed to converge independently of the initial guess and time-step size when $\mathcal L\geq M_G=\max_{0\leq\xi\leq 1}(-\partial_1G_-(\xi,T,c))$.
\end{theorem}
\begin{proof}
To show when the iterations (\ref{eq:ACLscheme}) converge, we follow similar steps as \cite[Lemma 4.1]{kumar2014convergence} and we show the main steps in the following. We first define the error $e^{j+1}:= \phi^{n+1,j+1}-\phi^{n+1}$. By subtracting (\ref{eq:ACtimepm}) from (\ref{eq:ACLscheme}), we get
\begin{equation}
e^{j+1}(1+\Delta t\mathcal L) - \gamma\Delta t\nabla^2e^{j+1} = \Delta t(G_-(\phi^{n+1,j},T,c)-G_-(\phi^{n+1},T,c)) + \Delta t\mathcal Le^j. \label{eq:ACLerror}
\end{equation}
By the mean value theorem we have that $G_-(\phi^{n+1,j},T,c)-G_-(\phi^{n+1},T,c)=\partial_1G_-(\xi,T,c)e^j$, for some number $\xi$ between $\phi^{n+1,j}$ and $\phi^{n+1}$. We define
\begin{equation}
M_G = \max_{0\leq\xi\leq 1}(-\partial_1G_-(\xi,T,c)).\label{eq:MF}
\end{equation}
Since $G_-(\phi,T,c)$ is monotonically decreasing with respect to $\phi$, $\partial_1G_-(\phi,T,c)$ is either negative or zero; hence, $M_G$ is a positive number. We multiply (\ref{eq:ACLerror}) with $e^{j+1}$ and integrate over the domain $\Omega$. Using our knowledge about $G_-(\phi,T,c)$, and letting $\mathcal L \geq M_G$, this can be written as
\begin{equation*}
\|e^{j+1}\|^2(1+\Delta t\mathcal L) + \gamma\Delta t\|\nabla e^{j+1}\|^2 \leq \Delta t \mathcal L\|e^{j}\|\|e^{j+1}\|,  
\end{equation*}
where we in the second term have used integration by parts and on the right-hand side applied the Cauchy-Schwartz inequality as well as that $0\leq \partial_1G_-(\xi,T,c)+\mathcal L\leq \mathcal L$. The norms are $L^2$-norms. The right-hand side can be further rewritten using Young's inequality, hence
\begin{equation*}
\|e^{j+1}\|^2(1+\Delta t\mathcal L) + \gamma\Delta t\|\nabla e^{j+1}\|^2 \leq \frac{1}{2} \Delta t \mathcal L(\|e^{j}\|^2+\|e^{j+1}\|^2).
\end{equation*}
By collecting all terms with $e^{j+1}$ on the left-hand side, we arrive at
\begin{equation*}
\|e^{j+1}\|^2 + \frac{2\gamma\Delta t}{2+\Delta t\mathcal L}\|\nabla e^{j+1}\|^2 \leq \frac{ \Delta t \mathcal L}{2+\Delta t\mathcal L}\|e^{j}\|^2.
\end{equation*}
From this we can conclude that the L-scheme (\ref{eq:ACLscheme}) is a contraction and hence converges by the Banach fixed point theorem when
\begin{equation}
 \mathcal L\geq M_G.\label{eq:Lschemebounds}
\end{equation}
\end{proof}

Note that (\ref{eq:Lschemebounds}) is not a strict limit, and convergence can also be achieved under milder constrictions on $\mathcal L$, but will then have a restriction on $\Delta t$. This has been exemplified for e.g.~the Richards equation \cite{list2016iterative}. However, under the above constraints convergence is guaranteed, for any starting point. Note that the factor
\begin{equation}\label{eq:Lschemerate}
	L= \frac{ \Delta t \mathcal L}{2+\Delta t\mathcal L}
\end{equation}
can estimate how fast the convergence is. A smaller value of $L$ (closer to 0) means that convergence is expected to be faster. Hence, when $\mathcal L$ and/or $\Delta t$ grows, convergence is expected to be slower as $L$ would approach 1. This is in contrast to non-linear equations where the non-linearity is occurring in the time derivative. In this case, the L-scheme iterations are expected to be faster for larger time-step sizes $\Delta t$ \cite[Remark 4]{illiano2021iterative}.

In practice, the L-scheme iterations are performed until
\begin{equation}
	\|\phi^{n+1,j+1}-\phi^{n+1,j}\|\leq \text{tol}_{\mathcal L}, \label{eq:Lschemetol}
\end{equation}
where $\text{tol}_{\mathcal L}$ is a prescribed tolerance value. The newest $\phi^{n+1,j+1}$ is taken as the solution $\phi^{n+1}$ and the simulation can proceed.

\subsubsection{Determining $M_G$ and $G_-(\phi)$}
A relevant question is of course how to determine $M_G$ and $G_-(\phi)$ to perform the L-scheme iterations. As only an upper bound for $M_G$ is needed, we start with this. We search for $M_G$ such that
\begin{equation*}
M_G = \max_{0\leq\phi\leq 1}|\partial_1G(\phi,T,c)|.
\end{equation*}
This way, $M_G$ is potentially overestimated as the positive values of $\partial_1G(\phi,T,c)$ are also included, but this $M_G$ will clearly also fulfill (\ref{eq:MF}). Differentiating $G(\phi,T,c)$ with respect to $\phi$ gives
\begin{equation*}
\partial_1G(\phi,T,c) = \frac{\gamma}{\lambda^2}16\Big(-(1-6\phi+6\phi^2) + \frac{1}{|\Omega|}\int_{\Omega}(1-6\phi+6\phi^2)dV\Big) - \frac{4}{\lambda}\frac{1}{m_m}f(T,c)(1-2\phi).
\end{equation*}
Since the average of 1 is 1, this can be further simplified to
\begin{equation}
\partial_1G(\phi,T,c) = \frac{96\gamma}{\lambda^2}\Big(\phi(1-\phi)-\frac{1}{|\Omega|}\int_{\Omega}\phi(1-\phi)dV\Big) - \frac{4}{\lambda}\frac{1}{m_m}f(T,c)(1-2\phi)\label{eq:Fderiv}
\end{equation}
We then take the maximum over $0\leq\phi\leq 1$:
\begin{align*}
M_G &= \max_{0\leq\phi\leq 1}|\partial_1G(\phi,T,c)|\\ &\leq \max_{0\leq\phi\leq 1} \frac{96\gamma}{\lambda^2}|\phi(1-\phi)-\frac{1}{|\Omega|}\int_{\Omega}\phi(1-\phi)dV| + \frac{4}{\lambda}\frac{1}{m_m}|f(T,c)|\\ &\leq \frac{24\gamma}{\lambda^2} + \frac{4}{\lambda}\frac{1}{m_m}|f(T,c)|.
\end{align*}
Hence, the value of $M_G$ can be seen as an interplay between two processes: One due to changes in $\phi$ coming from curvature-driven motion, and one coming from chemical reactions. The latter could be zero, but $M_G$ will always be larger than zero since $\gamma,\lambda>0$.

The equation (\ref{eq:Fderiv}) can also be used to determine $G_-(\phi,T,c)$. Since $G_-(\phi,T,c)$ simply consists of the part of $G(\phi,T,c)$ where $G(\phi,T,c)$ is decreasing with respect to $\phi$, the sign of $\partial_1G(\phi,T,c)$ in (\ref{eq:Fderiv}) gives us the increasing and decreasing parts. However, we have to make sure that a splitting of $G$ into $G_-$ and $G_+$ and corresponding implicit/explicit time stepping still result in a conservative scheme.

\subsubsection{Strategy for L-scheme iterations for the conservative Allen-Cahn equation}
As pointed out in Section \ref{sec:discreteAC}, the discrete version (\ref{eq:discAC}) of the conservative Allen-Cahn equation (\ref{eq:rephase}) is still conservative as long as the same time step $\ell$ are used for same $\phi_K^\ell$ in both $P'(\phi_K^\ell)$ and in $\sum_{J\in\mathcal E}|J|P'(\phi_J^\ell)$, although $\ell=n$ or $\ell=n+1$ could be chosen independently for each control volume. Hence, when splitting the non-linearities into an increasing and decreasing part, this has to be done for control volume by control volume. We outline this in practice:

At time step $n$, $\phi^n_K$ is known for all $K\in\mathcal E$, and we use the space-discrete version of (\ref{eq:Fderiv}), namely
\begin{align*}
\partial_1G(\phi_K;\{\phi\}_{J\in\mathcal E},T_K,c_K) = &\frac{96\gamma}{\lambda^2}\Big(\phi_K(1-\phi_K)-\frac{1}{|\Omega|}\sum_{J\in\mathcal E}|J|\phi_J(1-\phi_J)\Big)\\ &- \frac{4}{\lambda}\frac{1}{m_m}f(T_K,c_K)(1-2\phi_K),
\end{align*}
to determine whether $\phi_K$ is causing $G(\phi_K;\{\phi\}_{J\in\mathcal E},T_K,c_K)$ to be increasing or decreasing. Hence, depending on the sign, $\phi_K$ will either be solved explicitly or implicitly (that is, either part of $G_+$ or $G_-$), and is assigned to time step $n$ or $n+1$. We denote this as $\phi_K^{\ell_K}$ to denote that $\ell_K$ is $n$ or $n+1$ in an element-wise manner. When all control volumes are checked and assigned to being either $n$ or $n+1$ in the non-linearity, the discrete L-scheme iterations are:
\begin{align}\label{eq:ACdiscLscheme}
\frac{1}{\Delta t} (\phi^{n+1,j+1}_K-\phi^n_K) = &\frac{\gamma}{|K|}\sum_{L\in \mathcal N(K)}|\sigma_{K,L}|F^{\phi;n+1,j+1}_{K,L}+ G_-(\phi^{n+1,j}_K,T_K,c_K)\nonumber\\ &+G_+(\phi^n_K,T_K,c_K)- \mathcal L(\phi^{n+1,j+1}_K-\phi^{n+1,j}_K),
\end{align}
where $G_-(\phi_K^{n+1,j},T_K,c_K)$ is $G(\phi_K^{n+1,j};\{\phi^{\ell_J}\}_{J\in\mathcal E},T_K,c_K)$ when $\phi_K^n$ was assigned to time step $n+1$, and $G_+(\phi_K^n,T_K,c_K)$ is $G(\phi_K^n;\{\phi^{\ell_J}\}_{J\in\mathcal E},T_K,c_K)$ when $\phi_K^n$ was assigned to time step $n$. Note that the status of each element is updated at every iteration to ensure that the correct splitting of the non-linearity is performed. The discrete version of (\ref{eq:Fnonlin}) accounting for the splitting is
\begin{align*}
G(\phi_K^{\ell_K};\{\phi^{\ell_J}\}_{J\in\mathcal E},T_K,c_K) = &-\frac{\gamma}{\lambda^2}P'(\phi_K^{\ell_K}) - \frac{4}{\lambda}\phi_K^{\ell_K}(1-\phi_K^{\ell_K})\frac{f(T_K,c_K)}{m_m}\\ &+ \frac{\gamma}{\lambda^2}\frac{1}{|\Omega|}\sum_{J\in\mathcal E}|J|P'(\phi_J^{\ell_J})
\end{align*}
where $\ell_K$ (and $\ell_J$) is either $n$ or $n+1,j$. Note that also $G_+(\phi_K^n,T_K,c_K)$ is updated throughout the L-scheme iterations, since the sum relying on $\{\phi^{\ell_J}\}_{J\in\mathcal E}$ is updated. Choosing the stabilization parameter $\mathcal L$ according to $M_G$, the scheme (\ref{eq:ACdiscLscheme}) converges to a solution that conserves the phase field. Note however, that since we in practice stop the L-scheme iterations when the error is within a certain threshold, see (\ref{eq:Lschemetol}), we will therefore expect an error in terms of conserving the phase field within that same threshold.

\subsection{Coupling the phase-field equation to the other model equations}\label{sec:coupmod}
The conservative Allen-Cahn equation (\ref{eq:rephase}) is coupled to (\ref{eq:phase}) describing flow and transport of solute and heat. The equations are tightly coupled: the reaction rate in (\ref{eq:rephase}) depends on the solute concentration and temperature, while the phase field appears in all equations in (\ref{eq:phase}). We here consider an iterative scheme, where at every time step each equation (\ref{eq:rephase}) and (\ref{eq:phase}) is solved separately, and we iterate between the equations until convergence before continuing to the next time step. The benefit of such a scheme is that each equation can be solved using separate strategies, according to the properties of each equation. For simplicity we here consider a case without flow and that $m_f = 2m_m$. Hence, solute and heat is transported with diffusion/conduction only, but where the diffusion and conduction through the fluid and mineral depends on the value of the phase field.

The solution strategy is as following, and similar to the strategies applied in \cite{BringedalBastidas3,brun2020iterative}. We discretize in time using a time-step size $\Delta t$, and let $\phi^{n+1},c^{n+1},T^{n+1}$ denote the phase field, concentration and temperature at the next time step. The time-discrete equations, when splitting the non-linearity in an explicit/implicit manner as needed for the phase field and solving the other equations fully implicitly, are

\begin{subequations}\label{eq:coupsallimp}
	\begin{align}
		\frac{1}{\Delta t}(\phi^{n+1}-\phi^n) &= \gamma\nabla^2\phi^{n+1} +G_{-}(\phi^{n+1},T^{n+1},c^{n+1})\nonumber\\ &\quad + G_+(\phi^n,T^{n+1},c^{n+1}),\label{eq:coupphaseimp} \\
		\frac{1}{\Delta t}(\phi^{n+1}c^{n+1} &+(1-\phi^{n+1})m_m -\phi^nc^n-(1-\phi^n)m_m)\nonumber\\ &= D\nabla(\phi^{n+1}\nabla c^{n+1}) \label{eq:coupsoluteimp}\\
		\frac{1}{\Delta t} (c_\phi^{n+1}\rho_\phi^{n+1}T^{n+1}-c_\phi^n\rho_\phi^nT^n) &= \nabla\cdot(k_\phi^{n+1}\nabla T^{n+1}).\label{eq:coupheatimp}
	\end{align}
\end{subequations}
We will solve these equations in an iterative manner. We let $i$ be the iteration index for the coupling iterations. The logic initialization of the coupling iterations is to take $\phi^{n+1,0}=\phi^n$, $c^{n+1,0}=c^n$ and $T^{n+1,0}=T^n$, where $\phi^n,c^n,T^n$ are the (already found) solutions from the previous time step. However, the iterations converge for any initial guess, which will be shown in proof of Theorem \ref{thm:coupling}. We first find the phase field $\phi^{n+1,i+1}$ by solving
\begin{align}\label{eq:coupphase}
	\frac{1}{\Delta t}&(\phi^{n+1,i+1}-\phi^n) = \gamma\nabla^2\phi^{n+1,i+1} +G_{-}(\phi^{n+1,i+1},T^{n+1,i},c^{n+1,i}) \nonumber\\ &+ G_+(\phi^n,T^{n+1,i},c^{n+1,i}) - \mathcal L_\text{coup}(\phi^{n+1,i+1}-\phi^{n+1,i}),
\end{align}
where the last term has been added as a stabilization term. This term helps us to ensure that the iterations are indeed converging and is needed for the proof of Theorem \ref{thm:coupling}. Note that the reaction rate is using the iteration index from the previous iteration for temperature and solute concentration, which are known. Hence, $\phi^{n+1,i+1}$ is the only unknown. This equation is still non-linear and will be solved with L-scheme iterations. See Remark \ref{remark:Lscheme} on how the L-scheme iterations are modified due to the presence of the newly added stabilization term. After solving for $\phi^{n+1,i+1}$, we find $c^{n+1,i+1}$ and $T^{n+1,i+1}$ by solving
\begin{subequations}\label{eq:coupsoluteheat}
	\begin{align}
	\frac{1}{\Delta t}&(\phi^{n+1,i+1}c^{n+1,i+1}+(1-\phi^{n+1,i+1})m_m-\phi^nc^n-(1-\phi^n)m_m)\nonumber\\ &= D\nabla(\phi^{n+1,i+1}\nabla c^{n+1,i+1}) \label{eq:coupsolute}\\
	\frac{1}{\Delta t}&((c_p\rho)_\phi^{n+1,i+1}T^{n+1,i+1}-(c_p\rho)_\phi^nT^n) = \nabla\cdot(k_\phi^{n+1,i+1}\nabla T^{n+1,i+1}).\label{eq:coupheat}
	\end{align}
\end{subequations}
Recall that $(c_p\rho)_\phi,k_\phi$ depend on $\phi$ and are here using $\phi^{n+1,i+1}$. Since $\phi^{n+1,i+1}$ is known, these two equations can be solved for $c^{n+1,i+1},T^{n+1,i+1}$. Both these equations are linear, hence they can be solved directly when applying a spatial discretization. Overall the scheme is as illustrated in Figure \ref{fig:scheme}.

\begin{remark}\label{remark:Lscheme}
Note that the L-scheme iterations are slightly altered as the term $\mathcal L_\text{coup}(\phi^{n+1,i+1}-\phi^{n+1,i})$ now appears in (\ref{eq:coupphase}). In the resulting L-scheme iterations, we will at every time step $n+1$ and at every coupling iteration $i+1$, solve for $\phi^{n+1,i+1,j+1}$
\begin{align}\label{eq:ACLschemecoup}
	\frac{1}{\Delta t} &(\phi^{n+1,i+1,j+1}-\phi^n) = \gamma\nabla^2\phi^{n+1,i+1,j+1}\nonumber \\ &+ G_-(\phi^{n+1,i+1,j},T^{n+1,i},c^{n+1,i})+G_+(\phi^n,T^{n+1,i},c^{n+1,i})\nonumber \\ &- \mathcal L(\phi^{n+1,i+1,j+1}-\phi^{n+1,j}) - \mathcal L_\text{coup}(\phi^{n+1,i+1,j+1}-\phi^{n+1,i}),
\end{align}
where $G_-$ and $G_+$ are the same as before. Performing the same steps as in the proof of Theorem \ref{thm:Lscheme} we now obtain
\begin{equation*}
\|e^{j+1}\|^2 + \frac{2\gamma\Delta t}{2+\Delta t\mathcal L+\Delta t\mathcal L_\text{coup}}\|\nabla e^{j+1}\|^2 \leq \frac{ \Delta t \mathcal L}{2+\Delta t\mathcal L +\Delta t\mathcal L_\text{coup}}\|e^{j}\|^2.
\end{equation*}
Hence, we still obtain a contraction for any $\mathcal L_\text{coup}>0$ and assuming $\mathcal L\geq M_G$. In fact, the factor $L=\frac{ \Delta t \mathcal L}{2+\Delta t\mathcal L +\Delta t\mathcal L_\text{coup}}$ is now smaller compared to (\ref{eq:Lschemerate}) and we can expect slightly faster convergence with the presence of this new stabilization term. The L-scheme iterations are made until 
\begin{equation*}
	\|\phi^{n+1,i+1,j+1}-\phi^{n+1,i+1,j}\|\leq \text{tol}_{\mathcal L}, 
\end{equation*}
and the newest $\phi^{n+1,i+1,j+1}$ is taken as $\phi^{n+1,i+1}$; that is, the phase field of the current coupling iteration. 
\end{remark}

\begin{figure}[h!]
	\begin{center}
	\includegraphics[width=0.5\textwidth]{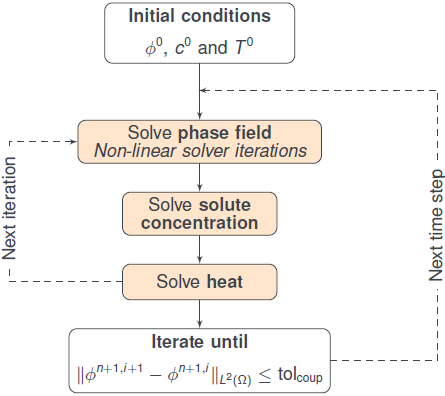}
	\caption{Scheme of coupling iterations}
	\label{fig:scheme}
	\end{center}
\end{figure}

To prove convergence of the coupling iterations, we rely on the following assumptions:
\begin{itemize}
	\item[(A1)] The phase field must be bounded away from 0 and 1.
	\item[(A2)] The solute concentration and temperature must be bounded.
	\item[(A3)] The gradient of the solute concentration and temperature must be bounded at every time step; that is, $\|\nabla c^n\|_{L^{\infty}(\Omega)}\leq C_c$ and $\|\nabla T^n\|_{L^{\infty}(\Omega)}\leq C_T$ for constants $C_c, C_T>0$ and for all time steps $n$.
\end{itemize}
The original Allen-Cahn equation can be shown to fulfill a maximum principle, also in the time-discrete form including stabilization terms, see \cite[Lemma 1]{BringedalBastidas3}. Including the non-local term does not alter this, hence also the conservative Allen-Cahn equation (\ref{eq:coupphase}) results in a phase field which is between 0 and 1. For a bounded domain and $\lambda>0$, the phase field will be bounded away from 0 and 1. Hence, (A1) is fulfilled. Additionally, \cite[Lemma 3]{BringedalBastidas3} proves the maximum principle for the solute concentration, using the same equation for the solute as considered here. The proof relies on the constraint 
\begin{equation}
4\gamma\leq \frac{\lambda k_0}{m_m}\label{eq:Manuelagamma}
\end{equation}
which we will obey as well when simulating the coupled model. The heat equation (\ref{eq:coupheat}) has the same structure as the solute equation (\ref{eq:coupsolute}), and hence can be shown to be bounded under same arguments as for the solute concentration. Hence, (A2) is fulfilled when the constraint is obeyed. Finally, (A3) is a reasonable assumption given the diffuse nature of the equations for solute concentration and temperature (\ref{eq:coupsoluteheat}). However, rigorous proofs of the assumptions (A1)-(A3) are beyond the scope of this manuscript. 

\begin{remark}
	Note that to obtain guarantees that the phase field is bounded away from 0 and 1, one can also regularize the phase-field equation with some value $\delta>0$, similar as has been done in \cite{BringedalBastidas3,bringedal2019phase}. In this case one rephrases the equation to using $\phi_\delta$, where $0<\delta\leq \phi_\delta\leq 1-\delta <1$. 
\end{remark}

\begin{theorem}\label{thm:coupling}
	Under the assumptions (A1)-(A3), the coupling iterations (\ref{eq:coupphase}), (\ref{eq:coupsoluteheat}) form a contraction and will hence converge for any initial guess when $\mathcal L_\text{coup}$ fulfills (\ref{eq:Lcouprestriction}) and the time-step size fulfills (\ref{eq:dtrestriction}).
\end{theorem}

\begin{proof}
We define the error terms
\begin{equation*}
	e^{i+1}_\phi := \phi^{n+1,i+1}-\phi^{n+1},\quad e^{i+1}_c := c^{n+1,i+1}-c^{n+1},\quad e^{i+1}_T := T^{n+1,i+1}-T^{n+1}.
\end{equation*}
We consider first the Allen-Cahn equation. By subtracting (\ref{eq:coupphase}) from the corresponding time-discrete form without decoupling, equation (\ref{eq:coupphaseimp}), we obtain
\begin{align*}
	\frac{1}{\Delta t}(\phi^{n+1,i+1}-\phi^{n+1}) = &\gamma \nabla^2(\phi^{n+1,i+1}-\phi^{n+1}) + G_-(\phi^{n+1,i+1},T^{n+1,i},c^{n+1,i})\\ &+ G_+(\phi^n,T^{n+1,i},c^{n+1,i})- G_-(\phi^{n+1},T^{n+1},c^{n+1})\\ &- G_+(\phi^{n},T^{n+1},c^{n+1}) - \mathcal L_\text{coup}(\phi^{n+1,i+1}-\phi^{n+1,i}).
\end{align*}
We multiply the above equation with a $\psi\in H^1_0(\Omega)$ and integrate over the domain $\Omega$. We then obtain
\begin{align}
	\frac{1}{\Delta t}(e^{i+1}_\phi,\psi) = &-\gamma (\nabla e^{i+1}_\phi,\nabla\psi ) + (G_-(\phi^{n+1,i+1},T^{n+1,i},c^{n+1,i}),\psi)\nonumber\\ &+ (G_+(\phi^n,T^{n+1,i},c^{n+1,i}),\psi)- (G_-(\phi^{n+1},T^{n+1},c^{n+1}),\psi)\nonumber\\ &- (G_+(\phi^{n},T^{n+1},c^{n+1}),\psi) -\mathcal L_\text{coup}(e^{i+1}_\phi-e^{i}_\phi,\psi),
\end{align}
where the notation $(\cdot,\cdot)$ refers to the usual $L^2$-inner product. We then let $\psi = e^{i+1}_\phi$, which gives us
\begin{align*}
	\|e^{i+1}_\phi\|^2 &+\Delta t\gamma \| \nabla e^{i+1}_\phi\|^2 + \Delta t\mathcal L_\text{coup}\|e^{i+1}_\phi\|^2 = \Delta t\mathcal L_\text{coup}(e^{i}_\phi,e^{i+1}_\phi)\\ &+ \Delta t(G_-(\phi^{n+1,i+1},T^{n+1,i},c^{n+1,i}),e^{i+1}_\phi)+ \Delta t(G_+(\phi^n,T^{n+1,i},c^{n+1,i}),e^{i+1}_\phi) \\ &- \Delta t(G_-(\phi^{n+1},T^{n+1},c^{n+1}),e^{i+1}_\phi) - \Delta t(G_+(\phi^{n},T^{n+1},c^{n+1}),e^{i+1}_\phi),
\end{align*}
We let $\mathcal M_G = \max(|\partial_1 G|,|\partial_2G|,|\partial_3G|)$, where we have maximized the partial derivatives of $G$ over the bounded domain where $\phi,T,c$ are defined (Assumptions (A1) and (A2)). Then, by the mean-value theorem, we get
\begin{align*}
(1+\Delta t\mathcal L_\text{coup})\|e^{i+1}_\phi\|^2 &+\Delta t\gamma \| \nabla e^{i+1}_\phi\|^2 \leq \Delta t\mathcal L_\text{coup}(e^{i}_\phi,e^{i+1}_\phi) \\ &  + \Delta t\mathcal M_G\|e^{i+1}_\phi\|^2 + 2\Delta t\mathcal M_G(e^i_T,e^{i+1}_\phi) + 2\Delta t \mathcal M_G (e^i_c,e^{i+1}_\phi).
\end{align*}
We apply Young's inequality to the first and the last two terms on the right-hand side, which, for $\delta_1,\delta_2,\delta_3>0$, gives
\begin{align*}
(1&+\Delta t\mathcal L_\text{coup} -\Delta t \mathcal M_G)\|e^{i+1}_\phi\|^2 +\Delta t\gamma \| \nabla e^{i+1}_\phi\|^2\\ \leq &\Delta t \mathcal L_\text{coup} \frac{\delta_1}{2}\|e^i_\phi\|^2 + \Delta t\mathcal L_\text{coup}\frac{1}{2\delta_1}\|e^{i+1}_\phi\|^2 +  \Delta t\mathcal M_G \delta_2\|e^i_T\|^2 + \Delta t\mathcal M_G\frac{1}{\delta_2}\|e^{i+1}_\phi\|^2 \\ &+  \Delta t\mathcal M_G \delta_3\|e^i_c\|^2 + \Delta t\mathcal M_G\frac{1}{\delta_3}\|e^{i+1}_\phi\|^2.
\end{align*}
We collect all terms with $\|e^{i+1}_\phi\|$ on the left-hand side and hence have
\begin{align}
	&(1+\Delta t\mathcal L_\text{coup} -\Delta t\mathcal M_G - \Delta t\mathcal L_\text{coup}\frac{1}{2\delta_1} -\Delta t\mathcal M_G \frac{1}{\delta_2}-\Delta t\mathcal M_G \frac{1}{\delta_3})\|e^{i+1}_\phi\|^2 \nonumber \\  &+\Delta t\gamma \| \nabla e^{i+1}_\phi\|^2 \leq \Delta t \mathcal L_\text{coup} \frac{\delta_1}{2}\|e^i_\phi\|^2  +  \Delta t\mathcal M_G \delta_2\|e^i_T\|^2  +  \Delta t\mathcal M_G \delta_3\|e^i_c\|^2. \label{eq:tempproof}
\end{align}
To proceed, we need bounds for $\|e^i_T\|$ and $\|e^i_c\|$, hence we turn to the temperature and solute concentration equation. 

For the solute concentration, we start by subtracting (\ref{eq:coupsolute}) from (\ref{eq:coupsoluteimp}), multiply with a test function from $H^1_0(\Omega)$ and integrate over the domain. Letting $e^{i+1}_c$ be the test function, we have
\begin{align*}
	\frac{1}{\Delta t}&(\phi^{n+1,i+1}e^{i+1}_c,e^{i+1}_c) + D(\phi^{n+1,i+1}\nabla c^{n+1,i+1},\nabla e^{i+1}_c) \\ &= D(\phi^{n+1}\nabla c^{n+1},\nabla e^{i+1}_c) + \frac{1}{\Delta t}(e^{i+1}_\phi(m_m-c^{n+1}),e^{i+1}_c).
\end{align*}
We use (A1); that $\phi_m\leq\phi$ for some $\phi_m>0$. Then the above equation can be rewritten to
\begin{align*}
	\frac{\phi_m}{\Delta t}&\|e^{i+1}_c\|^2 + D\phi_m\|\nabla e^{i+1}_c\|^2 \\ &\leq D((\phi^{n+1}-\phi^{n+1,i+1})\nabla c^{n+1},\nabla e^{i+1}_c) + \frac{1}{\Delta t}(e^{i+1}_\phi(m-c^{n+1}),e^{i+1}_c).
\end{align*}
We apply Young's inequality to the two terms on the right-hand side, which, for $\delta_4,\delta_5>0$, gives
\begin{align*}
\frac{\phi_m}{\Delta t}\|e^{i+1}_c\|^2 + D\phi_m\|\nabla e^{i+1}_c\|^2 \leq &D\frac{\delta_4}{2}\| e^{i+1}_\phi \nabla c^{n+1}\|^2 + D\frac{1}{2\delta_4}\|\nabla e^{i+1}_c\|^2\\ &+ \frac{1}{\Delta t}\frac{\delta_5}{2}\|(m-c^{n+1})e^{i+1}_\phi\|^2 + \frac{1}{\Delta t}\frac{1}{2\delta_5}\|e^{i+1}_c\|^2.
\end{align*}
We then apply (A2) and (A3), which results in
\begin{align*}
	(\frac{\phi_m}{\Delta t}-\frac{1}{\Delta t}\frac{1}{2\delta_5})\|e^{i+1}_c\|^2 + (D\phi_m-\frac{D}{2\delta_4})\|\nabla e^{i+1}_c\|^2 \leq  (D\frac{\delta_4}{2}C_c^2 +\frac{1}{\Delta t}\frac{\delta_5}{2}m_m^2)\|e^{i+1}_\phi\|^2
\end{align*}
To ensure positive terms on the left-hand side, we let $\delta_4 = \delta_5=\frac{1}{\phi_m}$. We then arrive at
\begin{equation*}
	\|e^{i+1}_c\|^2 \leq \frac{2\Delta t}{\phi_m} (\frac{DC_c^2}{2\phi_m}+\frac{ m_m^2}{2\phi_m\Delta t})\| e^{i+1}_\phi\|^2.
\end{equation*}
Note that this bound is also valid for index $i$. For simplicity, we rephrase the result with the more compact notation
\begin{equation}\label{eq:cboundphi}
	\| e^i_c\|^2 \leq E_c\| e^i_\phi\|^2,
\end{equation}
for $E_c= \frac{2\Delta t}{\phi_m} (\frac{DC_c^2}{2\phi_m}+\frac{ m_m^2}{2\phi_m\Delta t})$.

We perform similar steps for the temperature equation. We start by subtracting (\ref{eq:coupheat}) from (\ref{eq:coupheatimp}), multiply with a test function from $H^1_0(\Omega)$ and integrate over the domain. We let $e^{i+1}_T$ be the test function and hence obtain
\begin{align*}
	\frac{c_{p,f}\rho_f}{\Delta t}&(\phi^{n+1,i+1}e^{i+1}_T,e^{i+1}_T) + \frac{c_{p,m}\rho_m}{\Delta t}((1-\phi^{n+1,i+1})e^{i+1}_T,e^{i+1}_T) \\ &+ ((\phi^{n+1,i+1}k_f+(1-\phi^{n+1,i+1})k_m)\nabla T^{n+1,i+1},\nabla e^{i+1}_T) \\ = &((\phi^{n+1}k_f+(1-\phi^{n+1})k_m)\nabla T^{n+1},\nabla e^{i+1}_T)\\ &+ \frac{1}{\Delta t}(e^{i+1}_\phi(c_{p,m}\rho_m-c_{p,f}\rho_f)T^{n+1},e^{i+1}_T)
\end{align*}
We use (A1); that $\phi_m\leq\phi\leq \phi_M$ for $\phi_M<1$. Then the above equation can be rewritten to
\begin{align*}
	&\frac{c_{p,f}\rho_f\phi_m+c_{p,m}\rho_m(1-\phi_M)}{\Delta t}\|e^{i+1}_T\|^2 + (k_f\phi_m + k_m(1-\phi_M))\|\nabla e^{i+1}_T\|^2\\ &\leq ((k_f(\phi^{n+1}-\phi^{n+1,i+1})+k_m(\phi^{n+1,i+1}-\phi^{n+1}))\nabla T^{n+1},\nabla e^{i+1}_T)\\ &+ \frac{1}{\Delta t}(e^{i+1}_\phi(c_{p,m}\rho_m-c_{p,f}\rho_f)T^{n+1},e^{i+1}_T).
\end{align*}
We apply Young's inequality to the two terms on the right-hand side, which, for $\delta_6,\delta_7>0$, gives
\begin{align*}
		&\frac{c_{p,f}\rho_f\phi_m+c_{p,m}\rho_m(1-\phi_M)}{\Delta t}\|e^{i+1}_T\|^2 + (k_f\phi_m + k_m(1-\phi_M))\|\nabla e^{i+1}_T\|^2\\ &\leq \frac{\delta_6}{2}\| (-k_fe^{i+1}_\phi +k_me^{i+1}_\phi)\nabla T^{n+1}\|^2 + \frac{1}{2\delta_6}\|\nabla e^{i+1}_T\|^2\\ &+ \frac{1}{\Delta t}\frac{\delta_7}{2}\|(c_{p,m}\rho_m-c_{p,f}\rho_f)T^{n+1}e^{i+1}_\phi\|^2 + \frac{1}{\Delta t}\frac{1}{2\delta_7}\|e^{i+1}_T\|^2.
\end{align*}
We then apply (A2) and (A3), with $T_M$ being the upper bound of the temperature, which results in
\begin{align*}
	\frac{1}{\Delta t}&(c_{p,f}\rho_f\phi_m+c_{p,m}\rho_m(1-\phi_M)-\frac{1}{2\delta_7})\| e^{i+1}_T\|^2\\ &+ (k_f\phi_m+(1-\phi_M)k_m - \frac{1}{2\delta_6})\|\nabla e^{i+1}_T\|^2\\ &\leq (\frac{\delta_6}{2}|k_m-k_f|^2C_T^2 + \frac{1}{\Delta t}\frac{\delta_7}{2}|c_{p,m}\rho_m-c_{p,f}\rho_f|^2T_M^2) \|e^{i+1}_\phi\|^2
\end{align*}

To ensure positive terms on the left-hand side, we let $\delta_6 = \frac{1}{k_f\phi_m+k_m(1-\phi_M)}$ and $\delta_7=\frac{1}{c_{p,f}\rho_f\phi_m+c_{p,m}\rho_m(1-\phi_M)}$. We then arrive at
\begin{align*}
	&\frac{c_{p,f}\rho_f\phi_m+c_{p,m}\rho_m(1-\phi_M)}{2\Delta t}\|e^{i+1}_T\|^2\\ &\leq (\frac{\delta_6}{2}|k_m-k_f|^2C_T^2 + \frac{1}{\Delta t}\frac{\delta_7}{2}|c_{p,m}\rho_m-c_{p,f}\rho_f|^2T_M^2) \|e^{i+1}_\phi\|^2,
\end{align*}
where we, to keep the notation more compact, did not insert the chosen values for $\delta_6,\delta_7$ on the right-hand side. This bound is also valid for index $i$. For simplicity, we rephrase the result with the more compact notation
\begin{equation}\label{eq:Tboundphi}
	\| e^i_T\|^2 \leq E_T\| e^i_\phi\|^2,
\end{equation}
for $E_T = \frac{2\Delta t}{c_{p,f}\rho_f\phi_m+c_{p,m}\rho_m(1-\phi_M)}(\frac{\delta_6}{2}|k_m-k_f|^2C_T^2 + \frac{1}{\Delta t}\frac{\delta_7}{2}|c_{p,m}\rho_m-c_{p,f}\rho_f|^2T_M^2)$.

We then return to (\ref{eq:tempproof}) and insert here (\ref{eq:cboundphi}) and (\ref{eq:Tboundphi}). At the same time we let $\delta_1=1$ and $\delta_2=\delta_3=\frac{1}{2}$. Hence,
\begin{align}
	(1+&\frac{\Delta t}{2}\mathcal L_\text{coup} - 5\Delta t\mathcal M_G)\|e^{i+1}_\phi\|^2\nonumber\\ &\leq (\frac{\Delta t}{2}\mathcal L_\text{coup}+ \frac{\Delta t}{2}\mathcal M_G E_T + \frac{\Delta t}{2}\mathcal M_GE_c)\|e^{i}_\phi\|^2. \label{eq:contraction}
\end{align}
That means the coupling iterations form a contraction when the left-hand side is positive and when
\begin{equation*}
1+\frac{\Delta t}{2}\mathcal L_\text{coup} - 5\Delta t\mathcal M_G > \frac{\Delta t}{2}\mathcal L_\text{coup}+ \frac{\Delta t}{2}\mathcal M_G E_T + \frac{\Delta t}{2}\mathcal M_GE_c.
\end{equation*}
The latter translates to a restriction on the time-step size, and results in
\begin{equation}\label{eq:dtrestriction}
	\Delta t<\frac{1}{\mathcal M_G(5+E_c/2+E_T/2)}.
\end{equation}
To ensure that the left-hand side of (\ref{eq:contraction}) is positive, we need $\mathcal L_\text{coup}>10\mathcal M_G - \frac{2}{\Delta t}$, as well as $\mathcal L_\text{coup}>0$. Hence,
\begin{equation}\label{eq:Lcouprestriction}
	\mathcal L_\text{coup}>\max\{0,10\mathcal M_G - \frac{2}{\Delta t}\}.
\end{equation}
\end{proof}

Note that the time-step size restriction (\ref{eq:dtrestriction}) does not depend on the discretization mesh. However, it does depend on parameters that cannot easily be estimated. One possibility would be to use a coarse mesh and by trial-and-error find a time-step size that ensures convergence, which is then used on a finer mesh. However, the phase-field equation (\ref{eq:coupphase}) needs a certain mesh size to resolve the diffuse interface. A coarser mesh can be applied if the diffuse-interface width $\lambda$ is chosen larger, but changing the value of $\lambda$ can influence the time-step size restriction. Hence, one would for these model equations need to find the suitable time-step size on the mesh that is to be applied.

In practice, the coupling iterations are performed until
\begin{equation}
	\|\phi^{n+1,i+1}-\phi^{n+1,i}\|\leq \text{tol}_\text{coup},
\end{equation}
where $\text{tol}_\text{coup}$ is a prescribed tolerance. The newest $\phi^{n+1,i+1},c^{n+1,i+1},T^{n+1,i+1}$ is taken as $\phi^{n+1},c^{n+1},T^{n+1}$ and we can proceed to the next time step. As seen in the proof of Theorem \ref{thm:coupling}, the errors of the solute concentration and temperature are expected to be bounded when the error in the phase field is bounded. Hence, it is sufficient to consider a tolerance based on the phase field only. 
	
\begin{remark}
	Note that the value of $E_T$ in equation (\ref{eq:Tboundphi}) is zero if $k_m=k_f$ and $c_{p,f}\rho_f=c_{p,m}\rho_m$. This would correspond to equal heat conduction and heat storage properties in the fluid and mineral, which implies that the temperature equation is in practice independent of the phase field. A similar result can not be obtained for the solute.
\end{remark}

\section{Numerical results of the phase-field models}\label{sec:numbehavior}
We here perform numerical experiments to assess the numerical behavior of the conservative Allen-Cahn equation, in particular when it is solved with the L-scheme, and to test the full, coupled model to show its numerical behavior and potential. We investigate behavior of the Allen-Cahn equation only in Section \ref{sec:numAC}, before addressing the behavior of the coupled model in Section \ref{sec:numcoup}.

\subsection{Conservative and non-conservative Allen-Cahn}\label{sec:numAC}
We here consider the simplest test case possible, solving the original phase-field equation (\ref{eq:AC}) or the conservative version (\ref{eq:rephase}). The reaction rate will initially be set to zero, so only curvature-driven motion can change the shape (and, potentially, size) of the mineral. We consider two test cases, both to be considered in a two-dimensional domain $\Omega=(0,1)^2$, but with different initial conditions:
\begin{itemize}
	\item[(a)] We initialize the phase field as a circular mineral with an approximated diffuse interface transition at radius 0.3.
	\item[(b)] We initialize the phase field with ones everywhere except for in a centered square with side lengths 0.5 where the phase field is set to zero.
\end{itemize}  
The latter case (b) is more challenging as the initial condition violates the phase field equation and is more bound to numerical error. We solve the two test cases in three different ways:
\begin{itemize}
	\item[(i)] Solving (\ref{eq:AC}) fully implicitly, using Newton's method to solve the non-linear system of equations.
	\item[(ii)] Solving (\ref{eq:rephase}) fully implicitly, using Newton's method to solve the non-linear system of equations.
	\item[(iii)] Solving (\ref{eq:rephase}) semi-implcitly, using L-scheme to solve the non-linear system of equations as described in section \ref{sec:Lscheme}.
\end{itemize}
All approaches are solved using the Finite Volume discretization described in Section \ref{sec:FV}, on a uniform, rectangular grid. The test cases apply homogeneous Neumann conditions on the boundaries. For test case (a), only minor adjustments of the phase field diffuse interface is expected for the conservative approaches (ii) and (iii), as well as numerical error from the solvers. In test case (b), the curvature-driven motion of (ii) and (iii) will transform the initial square to a circular mineral. For the non-conservative phase field (i), we expect the mineral to shrink in both case (a) and (b).

All cases are solved on a $100\times100$ grid, with $\lambda=0.05$ and $\gamma=1$. Note that a smaller value of $\gamma$ causes the curvature-driven motion to be slower, but is here chosen large to highlight the issues with (\ref{eq:AC}). All equations are time stepped until $t=1$. The two fully-implicit cases (i) and (ii) are solved using $\Delta t = 10^{-4}$, which was needed to ensure convergence of the Newton iterations, and this time-step size is also applied for the L-scheme. 
With the current setup, we have $M_G=9.68\times 10^4$ and we choose $\mathcal L$ equal to this value. 
Newton iterations and L-scheme iterations are solved until a threshold of $10^{-13}$ is met in the $L^2$-norm between two subsequent iterations (that is, $\text{tol}_{\mathcal L}=10^{-13}$ in (\ref{eq:Lschemetol})). Since the size of the discrete system is not too large, the linear system is solved directly.

Table \ref{tab:deltavol} shows the change in volume of the mineral from initial condition until $t=1$ for the two test cases for all three approaches. Since the phase field approaches the value 0 in the mineral and 1 outside, the volume is calculated through
\begin{equation*}
\text{volume} = \int_{\Omega}(1-\phi)dV.
\end{equation*}

\begin{table}[htbp]
	{\footnotesize
		\caption{Volume change of mineral from first to last time step}  \label{tab:deltavol}
		\begin{center}
			\begin{tabular}{|l|c|c|c| } \hline
				Case/Approach & (i) & (ii) & (iii) \\ \hline
				(a)& -0.2850 & $2.0134\times10^{-11}$ & $1.5294\times 10^{-11}$  \\
				(b) & -0.25 & $-6.1989\times10^{-12}$ & $1.4864\times 10^{-11}$ \\ \hline
			\end{tabular}
		\end{center}
	}
\end{table}

Since all mineral disappeared for approach (i), the volume change corresponds to the loss of the initial amount of mineral. Note that there is a volume change due to the choice of the threshold tolerance. The two conservative approaches (ii) and (iii) show only change in volume that can be connected to the accuracy of the non-linear iterations. For approaches (ii) and (iii), at every time step a small volume change of about $1.5\times 10^{-15}$ is seen, which is the same as the tolerance of the Newton and L-scheme iterations.

Table \ref{tab:numit} shows average number of Newton or L-scheme iterations needed for each time step. Note however that in general a larger amount of iterations are needed in the beginning of the simulation and then stabilizes at a lower value. At least one Newton or L-scheme iteration is forced to be performed at every time step.

\begin{table}[htbp]
	{\footnotesize
		\caption{Average number of iterations for entire simulation}  \label{tab:numit}
		\begin{center}
			\begin{tabular}{|l|c|c|c| } \hline
				Case/Approach & (i) & (ii) & (iii) \\ \hline
				(a)& 1.03 & 1.0038 & 5.18  \\
				(b) & 1.03 & 1.0160 & 5.77\\ \hline
			\end{tabular}
		\end{center}
	}
\end{table}

For the approach (i), the entire mineral dissolves during the first 300 time steps and hence attains a constant solution with fluid only afterwards. 
When using the conservative scheme with Newton iterations; that is, approach (ii), 2-3 iterations are needed per time step in the beginning, which settles after 150 time steps. The Jacobian in the Newton iterations is however full due to the non-local term, which makes each Newton iteration expensive to solve.
For approach (iii), several L-scheme iterations are needed in the first phase while the diffuse interface (and for the square, also the shape) is adjusting. In the beginning 30-40 iterations are needed, before gradually decreasing to 3 L-scheme iterations after 1000 time steps, where it remains stable. However, the iterations are cheap as only a sparse linear system of equations needs to be solved. When solved on the same laptop and same software, approach (iii) was roughly a factor 2 faster than approach (ii).

Although the L-scheme iterations converge, they converge slowly. The size of $\mathcal L$ is responsible for this. Although we chose it according to the theoretical bounds (\ref{eq:Lschemebounds}), earlier studies of other equations have shown (see e.g. \cite{illiano2021iterative,list2016iterative,mitra2019lscheme}) that much lower values of $\mathcal L$ are possible, and will generally give faster convergence. Also larger values of $\Delta t$ could be used, but is for our equation expected to result in slower convergence for the L-scheme. Also note that lower values of $\mathcal L$ would usually lead to a time-step size restriction to obtain convergence. Hence, in the following two subsection we investigate numerically possible values of $\mathcal L$ and $\Delta t$.

\subsubsection{Size of $\mathcal L$}\label{sec:varyL}
We change the setup from the previous test slightly to not only consider semi-steady-state solutions. Initially we have a circular mineral of radius 0.3 as in (a), but this time with a reaction rate $f = -0.1$. Hence the mineral gradually dissolves, but is still present at $t=1$. We use $\gamma=0.1$ to have a more typical value for the curvature-driven motion. The value of $M_G$ is hence also changed and is $M_G=968$ now. 
We will for now fix $\Delta t=10^{-3}$ and vary $\mathcal L$. Table \ref{tab:numitL} shows the average number of iterations in each time step as a function of $\mathcal L$. Note that the numbers are generally larger than in Table \ref{tab:numit} as no steady-state is reached during the simulation. In this setup, the number of L-scheme iterations per time step is almost constant throughout the simulation.
We see how significantly lower values than the theoretical bound of $\mathcal L$ still provide convergence, and that the amount of iterations needed decreases when $\mathcal L$ decreases, before eventually increasing. Hence, there appears to be an optimal $\mathcal L$ for when the lowest amount of needed L-scheme iterations is obtained. For other equations, it has been shown that a lower value of $\mathcal L$ provides faster convergence (i.e., less iterations needed), and that there is an optimal $\mathcal L$ where the least amount of iterations are needed \cite{list2016iterative}, and $\mathcal L$ can be chosen differently from time step to time step to optimize convergence \cite{mitra2019lscheme}. There is also a lower limit for $\mathcal L$ which can still ensure convergence if the time-step size is also large \cite{list2016iterative}.

\begin{table}[htbp]
	{\footnotesize
		\caption{Average number of L-scheme iterations per time step for different values of $\mathcal L$}  \label{tab:numitL}
		\begin{center}
			\begin{tabular}{|l|c|c|c|c| } \hline
				Value of $\mathcal L$ & $M_G$ & $M_G/2$ & $M_G/4$ & $M_G/8$ \\ \hline
				Average number of iterations & 32 & 21 & 17 & 24  \\ \hline
			\end{tabular}
		\end{center}
	}
\end{table}

\subsubsection{Size of $\Delta t$}
We now use either the theoretical $\mathcal L=M_G$ or the $\mathcal L=M_G/4$ as this gave the lowest value of iterations, and vary the time-step size $\Delta t$. The setup is otherwise the same as in Section \ref{sec:varyL}. Table \ref{tab:numitdt} and \ref{tab:numitdt4} show the average number of iterations as a function of the time-step size for these two choices of $\mathcal L$. As expected for our equation, the number of iterations needed at every time step increases with the time-step size $\Delta t$. One case is deemed to not converge as the threshold value was not reached within 200 L-scheme iterations per time step. Note that convergence potentially could be achieved by allowing more iterations. Note that for $\mathcal L=M_G$, the number of L-scheme iterations increases but is less than doubled as the time-step size is doubled. Hence, in this case there is a computational gain by increasing the time-step size as the total number of L-scheme iterations throughout the simulation decreases. This is however not necessarily the case for $\mathcal L=M_G/4$ as the average number of L-scheme iterations is more than doubled as the time-step size increases from $2\times 10^{-3}$ to $4\times 10^{-3}$.

\begin{table}[htbp]
	{\footnotesize
		\caption{Average number of L-scheme iterations per time step for different values of $\Delta t$, when $\mathcal L = M_G$}  \label{tab:numitdt}
		\begin{center}
			\begin{tabular}{|l|c|c|c|c| } \hline
				Value of $\Delta t$ & $10^{-3}$ & $2\times10^{-3}$ & $4\times10^{-3}$ & $8\times10^{-3}$ \\ \hline
				Average number of iterations & 32 & 53 & 95 & 175  \\ \hline
			\end{tabular}
		\end{center}
	}
\end{table}

\begin{table}[htbp]
	{\footnotesize
		\caption{Average number of L-scheme iterations per time step for different values of $\Delta t$, when $\mathcal L = M_G/4$}  \label{tab:numitdt4}
		\begin{center}
			\begin{tabular}{|l|c|c|c|c| } \hline
				Value of $\Delta t$ & $10^{-3}$ & $2\times10^{-3}$ & $4\times10^{-3}$ & $8\times10^{-3}$ \\ \hline
				Average number of iterations & 17 & 30 & 83 & -  \\ \hline
			\end{tabular}
		\end{center}
	}
\end{table}

\subsection{Behavior of coupled model}\label{sec:numcoup}
We here consider the coupled model formulated in Section \ref{sec:coupmod}. As example case for the simulations, we consider a channel $\Omega=(0,1)^2$, where a mineral layer of initial thickness $0.25$ fills the lower part of the domain. Temperature is initially 1 everywhere and the fluid has a solute concentration of $0.5$, which corresponds to chemical equilibrium of the chosen reaction rate. A Dirichlet boundary condition of $T=0.9$ and $c=0.25$ is applied to the left boundary, while all other boundaries have zero Neumann boundary conditions. The phase field has zero Neumann boundary conditions on all boundaries. The lowering of the solute concentration (and temperature) at the left boundary is expected to trigger dissolution of the mineral and hence an increase in regions where the phase field approaches one. The chosen reaction rate is a simpler version of (\ref{eq:fpfd}) with constant solubility product. Hence,
\begin{equation*}
f(T,c) = f_p(T,c)-f_d(T,c) = k_0e^{-\frac{E}{RT}}\Big(\frac{c^2}{c_\text{eq}^2}-1\Big).
\end{equation*}
We here consider for simplicity a constant equilibrium concentration $c_\text{eq}=0.5$, but one could consider also temperature-dependent solubility \cite{plummer1988computer}.  
All parameters appearing in the model equations are as specified in Table \ref{tab:parameters}. We have not attempted to choose physically correct parameters, but generally chosen values around 1 to mimic a non-dimensional setup. As we need to fulfill (\ref{eq:Manuelagamma}), we have adjusted $\gamma$. For the heat conductivities we choose a larger heat conductivity in the mineral to be able to observe some heterogeneous influence. The simulation setup is sketched in Figure \ref{fig:simsetup}.

\begin{table}[htbp]
	{\footnotesize
	\caption{Parameter choices} \label{tab:parameters}
	\begin{center}
		\begin{tabular}{|l|c|}\hline
			Parameter & Value \\ \hline
			$\lambda$ & 0.05 \\ 
			$\gamma$ & 0.01  \\ \hline
			$E/R$ & 1 \\ 
			$k_0$ & 1 \\
			$c_\text{eq}$ & 0.5 \\ \hline
			$m_m$ & 1 \\ 
			$D$ & 1 \\ \hline 
			$\rho_f$ & 1 \\ 
			$c_{p,f}$ & 1 \\
			$\rho_m$ & 1 \\ 
			$c_{p,f}$ & 1 \\ 
			$k_f$ & 1 \\ 
			$k_m$ & 2 \\ \hline
		\end{tabular}
	\end{center}
	}
\end{table}

\begin{figure}[h!]
	\begin{center}
		\begin{overpic}[width=0.7\textwidth]{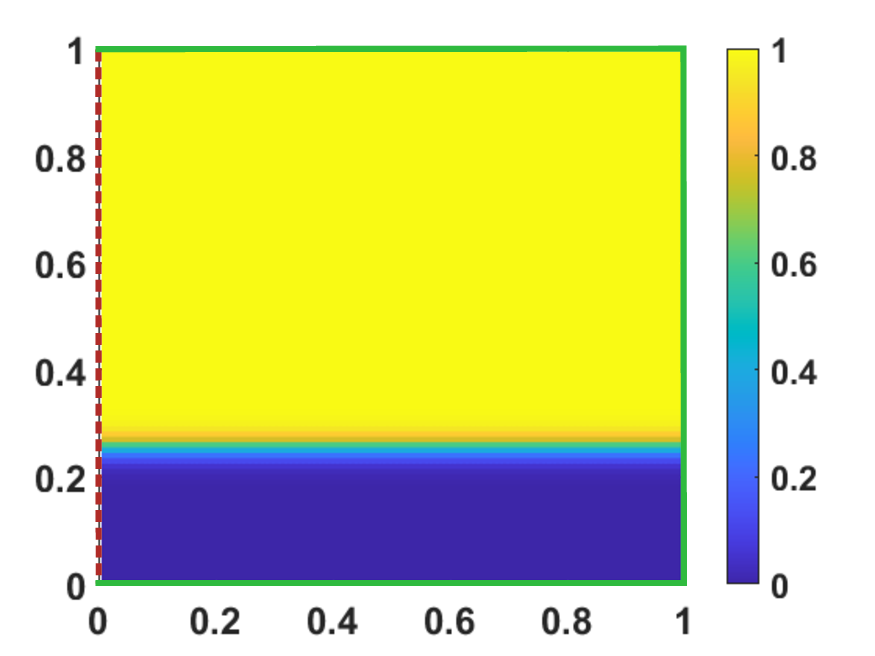}
			\put(38,42){$T^0=1$}
			\put(38,37){$c^0=0.5$}
			\put(13,24){\rotatebox{90}{$T=0.9, c = 0.25$}}
			\put(20,71){$\partial_n\phi=0,\partial_nT=0,\partial_nc=0$}
		\end{overpic}
	\caption{Initial phase field $\phi^0$ (background color), with initial temperature and solute concentration in text. Left boundary (dashed line) has Dirichlet boundary conditions for temperature and solute, while all other boundaries (solid line) have homogeneous Neumann boundary conditions. The phase field has homogeneous Neumann boundary conditions on the left boundary.}
	\label{fig:simsetup}
	\end{center}
\end{figure}
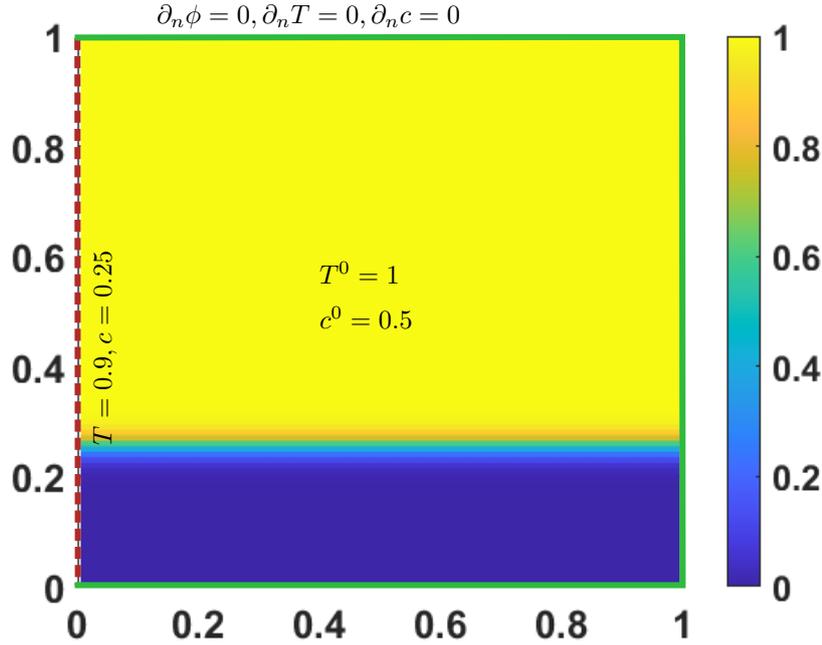

\subsubsection{Size of time-step size, tolerances and stabilization parameters}
We here apply a slightly larger tolerance for the L-scheme iterations than in Section \ref{sec:numAC}, but keep it lower than the tolerance for the coupling iterations to ensure that errors from the non-linear solver scheme does not influence the coupling iterations error. We therefore use $\text{tol}_{\mathcal L}=10^{-8}$ and $\text{tol}_\text{coup}=10^{-6}$ throughout all simulations presented in the following. Lower tolerances generally lead to more iterations.

With our parameter choices, we have $M_G = 118$, and we choose $\mathcal L=M_G$ in all simulations. For the coupling parameter $\mathcal L_\text{coup}$ we choose it small as this provide faster convergence for the coupling iterations. This result coincides with similar investigations (e.g.~\cite{BringedalBastidas3,brun2020iterative}). We here present results with $\mathcal L_\text{coup}=10^{-4}$. Although we may not fulfill (\ref{eq:Lcouprestriction}), numerical testing show that the coupling iterations converge for any $\mathcal L_\text{coup}>0$, and even with $\mathcal L_\text{coup}=0$.

By testing various time-step sizes, we found time-step sizes even up to $10^{-2}$ to provide converging coupling iterations. However, such long time-step sizes also resulted in very slowly converging L-scheme iterations, with generally more than 200 L-scheme iterations per coupling iteration needed. Hence, in the following, only results from shorter time-step sizes are presented. Although we cannot guarantee that (\ref{eq:dtrestriction}) (and (\ref{eq:Lcouprestriction})) is fulfilled, the coupling iterations were always found to converge fast for the investigated parameters.

\subsubsection{Behavior of coupling iterations}
For $\Delta t=10^{-3}$, $\mathcal L_\text{coup}=10^{-4}$ and $\mathcal L=M_G = 118$, we need an average of 2.21 coupling iterations per time step. More iterations are needed initially as the changes in solute concentration and temperature, and hence in phase field, are larger due to the boundary condition on the left boundary. For the L-scheme iterations, we need 6.00, 3.22 and 2.08 iterations in the first, second and third coupling iterations, respectively. See Figure \ref{fig:nint} for the evolution of iterations over time. Fewer L-scheme iterations are needed in later coupling iterations as we use the solution found in the previous coupling iteration when starting the L-scheme in later coupling iterations, giving a starting point that is already closer to the converged solution. Compared to Section \ref{sec:numAC}, fewer L-scheme iterations are needed for same time-step size as the tolerance for the L-scheme iterations is now larger.

\begin{figure}[h!]
	\begin{center}
	\includegraphics[width=0.45\textwidth]{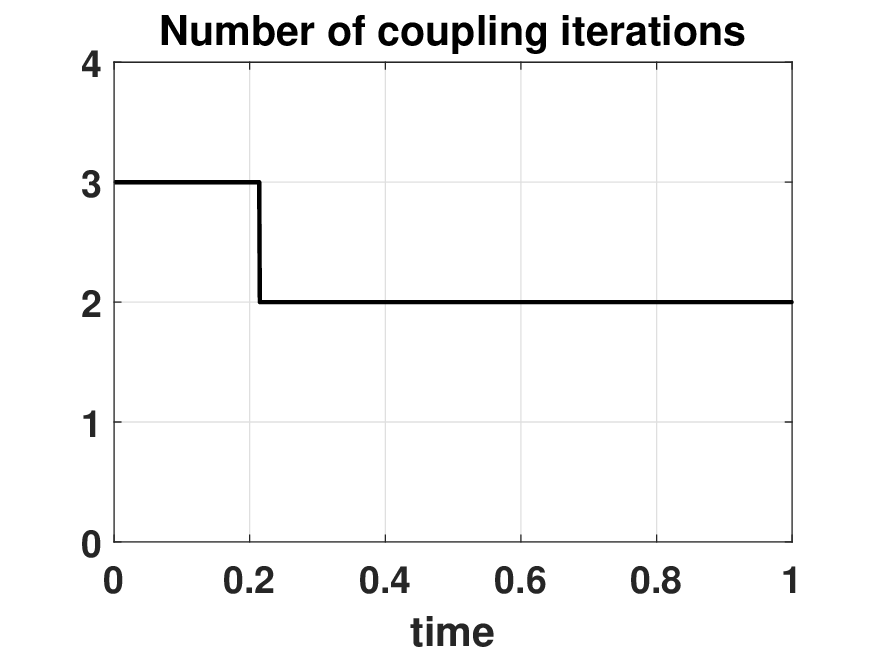}
	\includegraphics[width=0.45\textwidth]{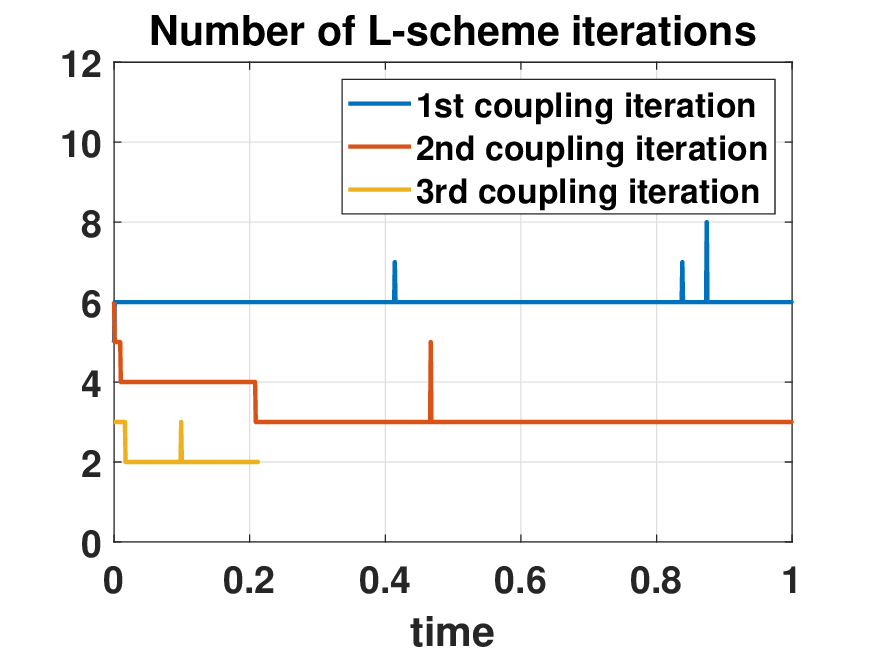}
	\caption{Number of coupling iterations (left) and number of L-scheme iterations per coupling iteration (right) over time for $\Delta t=10^{-3}$, $\mathcal L_\text{coup}=10^{-4}$ and $\mathcal L=M_G$.}
	\label{fig:nint}
\end{center}
\end{figure}

When increasing the time-step size, the number of coupling iterations increases slightly, see Table \ref{tab:numcoup}. However, the number of L-scheme iterations that are needed per coupling iterations increases significantly. For $\Delta t=8\times 10^{-3}$, more than 200 L-scheme iterations per coupling iteration are needed. Hence, we did not proceed with larger time-step sizes.

\begin{table}[htbp]
	{\footnotesize
		\caption{Average number of coupling iterations per time step for different values of $\Delta t$}  \label{tab:numcoup}
		\begin{center}
			\begin{tabular}{|l|c|c|c|c| } \hline
				Value of $\Delta t$ & $10^{-3}$ & $2\times10^{-3}$ & $4\times10^{-3}$ & $8\times10^{-3}$ \\ \hline
				Average number of coupling iterations & 2.21 & 2.64 & 3.11 & 3.58  \\ \hline
			\end{tabular}
		\end{center}
	}
\end{table}

\subsubsection{Simulation results of coupled model}
The simulation results for phase field, temperature and solute concentration after 0.5 and 1 time unit are shown in Figure \ref{fig:resultsvaryT}. The results are for the solution found using $\Delta t=10^{-3}$, but the solutions are qualitatively the same for the various time-step sizes that have been investigated. As expected, the mineral layer dissolves, which corresponds to a smaller region where the phase field attains values close to 0. As we used different heat conductivities in the fluid and mineral, vertical variability in the temperature profile can be seen in the transition between fluid and mineral. At $t=1$, the temperature is close to the boundary condition $T=0.9$ in the entire domain, showing that the heat conduction has propagated throughout the domain. For the solute we visualize the product $\phi c$ to highlight the solute concentration in the fluid. The solute concentration shows a gradient near the mineral interface as more solute is released into the fluid by the mineral dissolution. Hence, the overall solute concentration in the domain approaches the boundary value $c=0.25$ rather slowly due to ions being released by the chemical reaction.

\begin{figure}[h!]
		\begin{center}
	\includegraphics[width=0.32\textwidth]{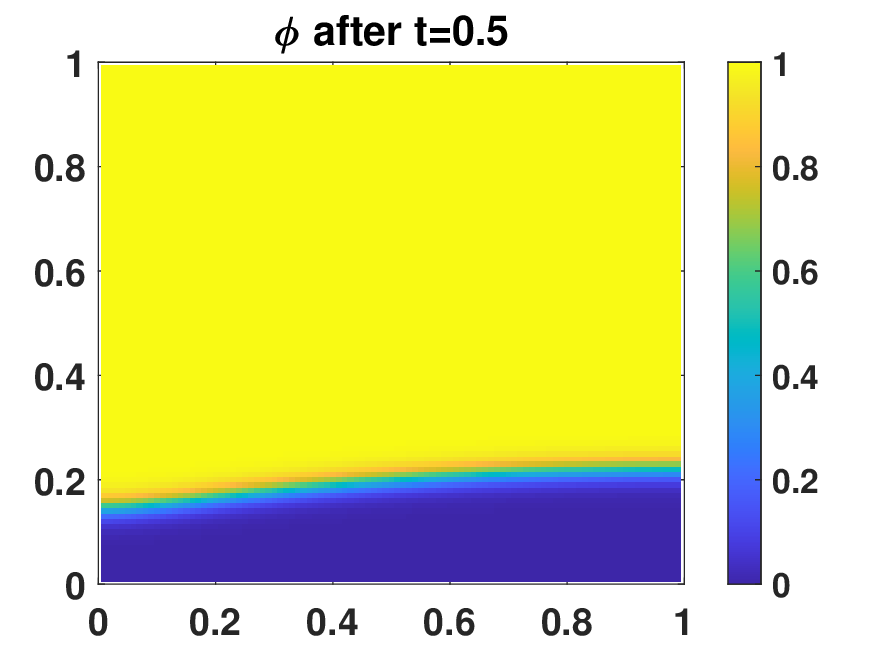}
	\includegraphics[width=0.32\textwidth]{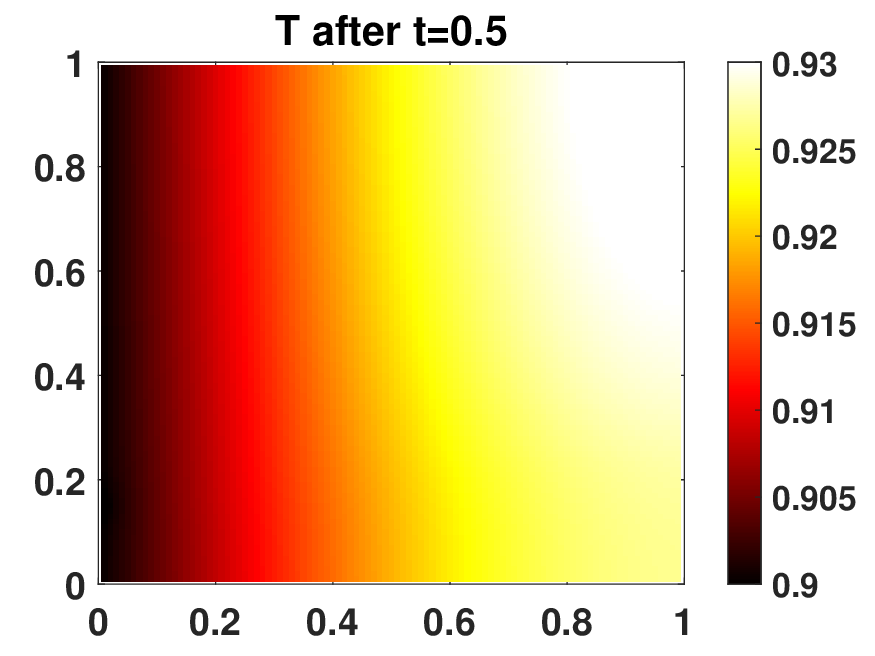}
	\includegraphics[width=0.32\textwidth]{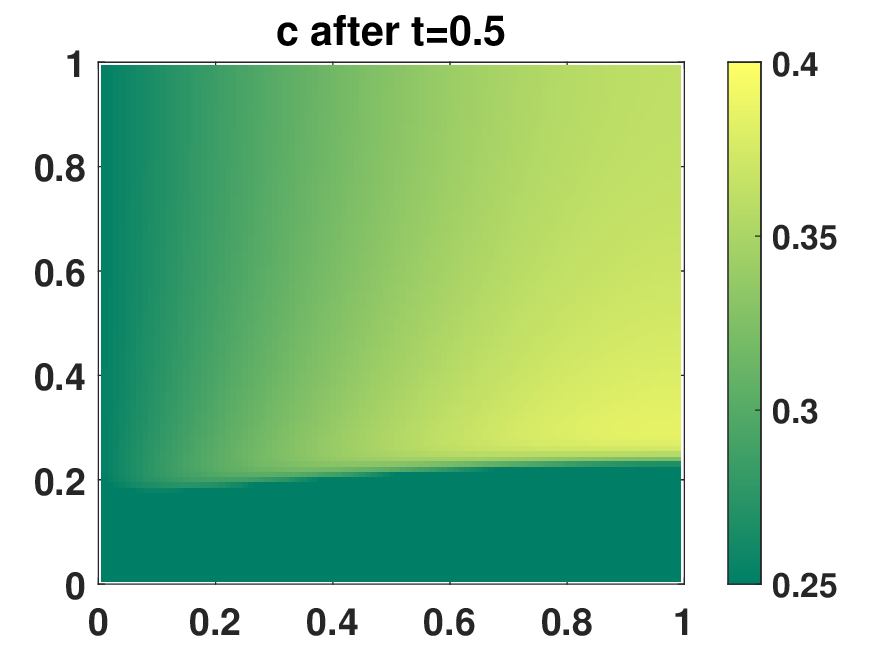}
	\includegraphics[width=0.32\textwidth]{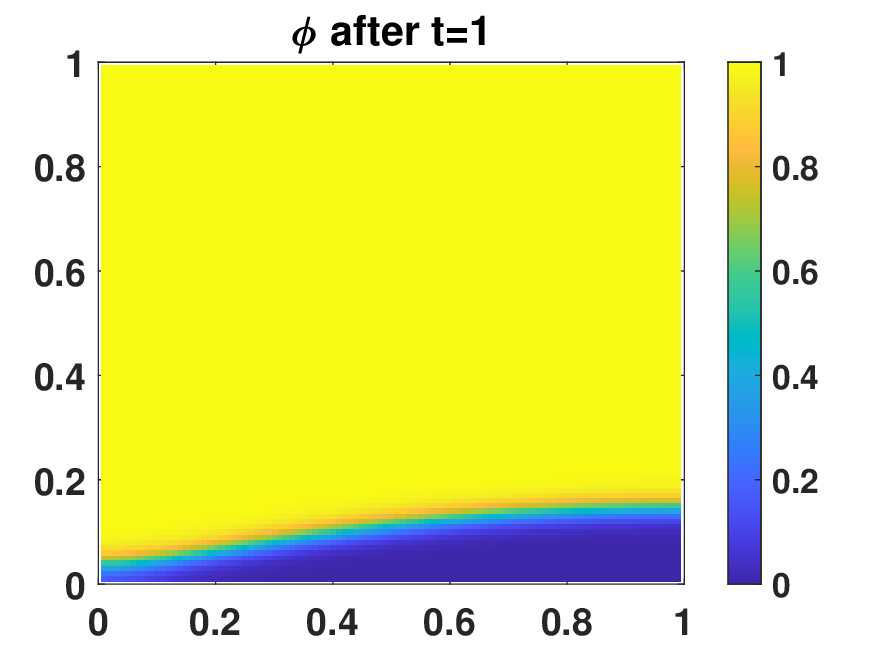}
	\includegraphics[width=0.32\textwidth]{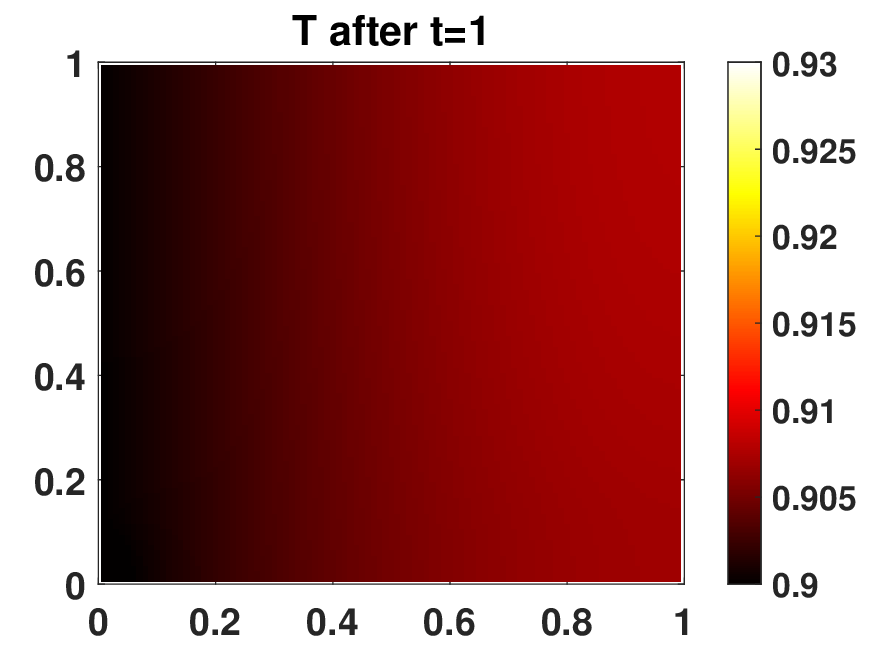}
	\includegraphics[width=0.32\textwidth]{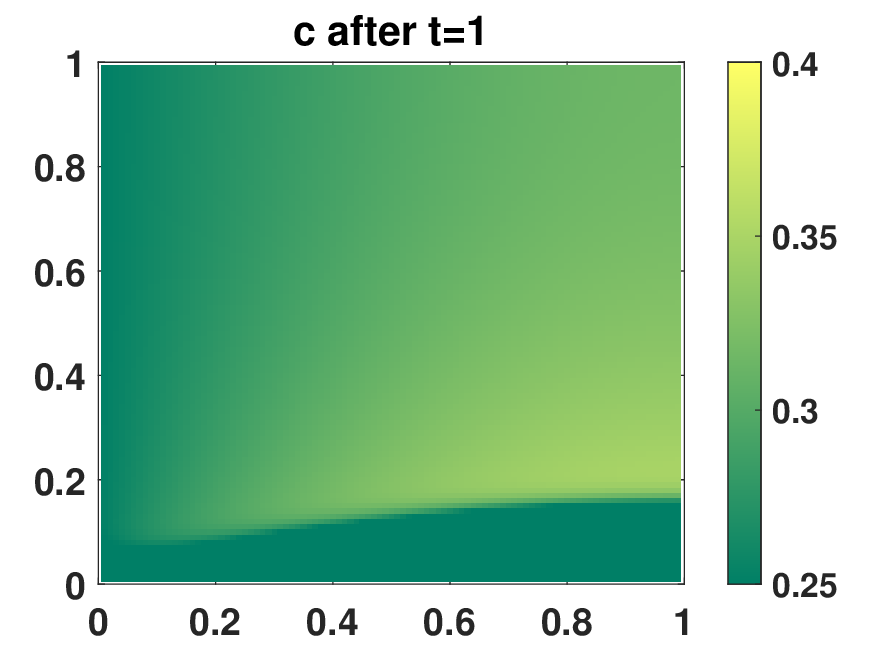}
	\caption{Phase field (left), temperature (middle) and solute concentration $\phi c$ (right) after $t=0.5$ (top) and $t=1$ (bottom).}
	\label{fig:resultsvaryT}
\end{center}
\end{figure}

We also calculate how the changes in the total amount of phase field, which corresponds to the amount of fluid, correspond to the reaction rate at every time step. That is, we calculate
\begin{equation*}
\phi_\text{int}^n=	\int_\Omega \phi^n dV\quad \Rightarrow \Delta\phi_\text{int}^n =\phi_\text{int}^{n+1}-\phi_\text{int}^n
\end{equation*}
for the changes in phase field, and
\begin{equation*}
	R^n = \Delta t\int_{\Omega}-\frac{4}{\lambda}\phi^n(1-\phi^n)\frac{1}{m_m}f(T^n,c^n)dV
\end{equation*}
to estimate the size of the reaction term at every time step. The resulting curves can be found in Figure \ref{fig:volumechange}. As observed here, the two lines are almost perfectly on top of each other, with a difference in the order of $10^{-7}$, which can be explained by the chosen tolerances for the coupling iterations and the L-scheme iterations. From Figure \ref{fig:volumechange} we also observe the size of the total reaction rate. Although temperature decreases, which corresponds to lower reaction rates and hence smaller change of $\phi$, the change of $\phi$ increases throughout the simulation as the decrease in solute concentration dominates.

\begin{figure}[h!]
		\begin{center}
	\includegraphics[width=0.45\textwidth]{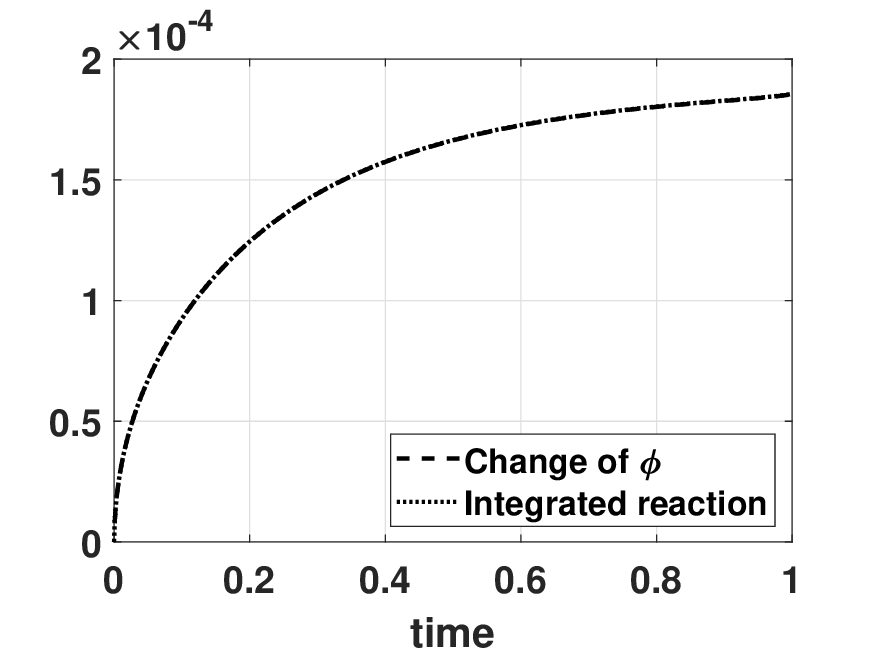}
	\includegraphics[width=0.45\textwidth]{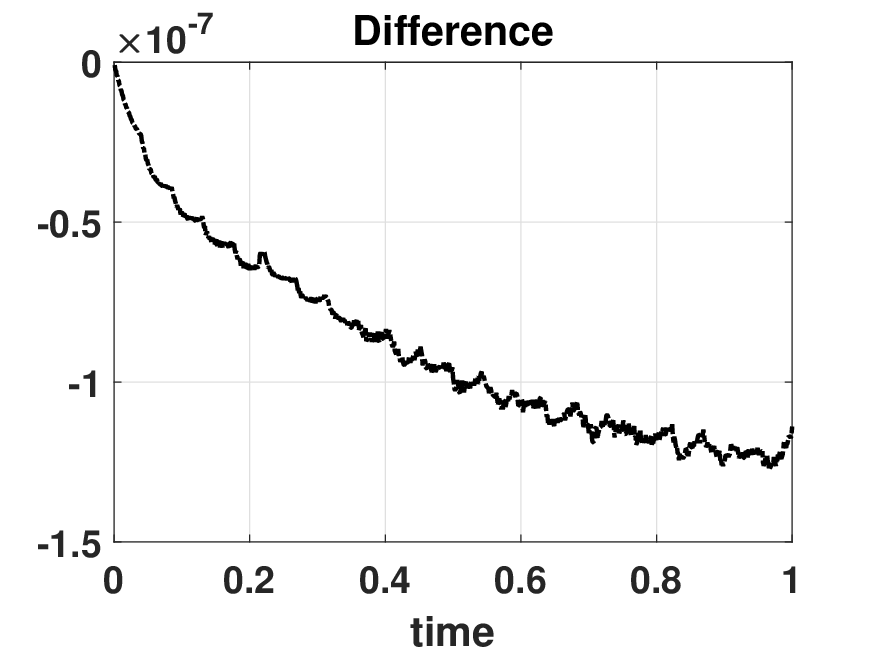}
	\caption{Changes in amount of phase field $\phi_\text{int}^n$ and integrated reaction rate $R^n$ (left), and the difference between these two curves (right) over time. The two lines in the left plot are almost perfectly on top of each other and cannot be visually separated.}
	\label{fig:volumechange}
\end{center}
\end{figure}

\section{Discussion and conclusion}\label{sec:conclusion}
In this paper we have formulated a phase-field model for reactive transport, where fluid flow, solute transport and heat transport is coupled with mineral precipitation and dissolution. The phase-field model is formulated rather general and can allow expansion/contraction effects due to the chemical reactions and is hence quasi-compressible. By taking the sharp-interface limit of the phase-field model, the model reduces to the expected sharp-interface model and recovers conservation of mass and energy in the fluid domain, conservation of energy in the mineral domain, as well as the expected exchanges of mass and energy across the evolving fluid-mineral interface. However, when using the original Allen-Cahn equation, the fluid-mineral interface itself evolves due to curvature-driven motion and would lead to non-physical loss or gain of mineral. We therefore consider a reformulated Allen-Cahn equation which is conservative. In this case the total phase field in the considered domain can only change due to chemical reactions in the domain, but encounters a conservative redistribution towards constant curvature of mineral. Mass and energy is also conserved.

The conservative Allen-Cahn equation includes a non-local term. Although easy to numerically discretize, the non-local term leads to a full Jacobian matrix if using Newton's method to solve the resulting non-linear system of equations, which makes Newton iterations expensive. We instead applied L-scheme iterations to solve the resulting non-linear system of equations. Although L-scheme iterations only offer linear convergence, they can be proven to always converge independently of time-step size and initial guess, when choosing the introduced stabilization parameter $\mathcal L$ appropriately. We could also show that the L-scheme iterations will converge to a solution that is discretely conservative. We note that another alternative to obtain a conservative phase-field model would be to use the Cahn-Hilliard equation instead of the conservative Allen-Cahn equation. 

To couple the conservative Allen-Cahn equation to the other model equations, we use an iterative scheme where the various model equations are solved sequentially. This enables us to solve each equation independently of the others and to optimize the solver scheme for each equation. We here consider a simplified setup without flow, hence only solute diffusion and heat conduction is present next to the phase-field evolution and chemical reactions. In this case, we can prove the convergence of the coupling iterations when adding a stabilization term $\mathcal L_\text{coup}$ to the phase-field equation, resulting in a restriction on the time-step size and on the stabilization term. 
For the future it would be beneficial to also include the flow in the coupling iterations. Incorporating the adapted Navier-Stokes equation in the proof of the convergence of the scheme is problematic, due to the occurring non-linearities and lack of reasonable bounds on the velocity. An alternative would be to use an adapted Stokes equation for the flow, which would be linear in velocity and offer the needed bounds. Using Stokes instead of Navier-Stokes is reasonable as long as the Reynolds number is small enough, which for applications in porous media is generally fulfilled.

The resulting scheme hence relies on good choices for the stabilization parameters $\mathcal L$ and $\mathcal L_\text{coup}$ as well as the time-step size $\Delta t$. Although the theoretical bound for $\mathcal L$ can easily be estimated, the coupling parameter $\mathcal L_\text{coup}$ and time-step size restriction of $\Delta t$ cannot easily be calculated. By numerical experiments we could show that convergence of the L-scheme iterations can be achieved for lower values of $\mathcal L$ than the estimated bound, which is a well-known behavior for L-scheme iterations. However, convergence of the L-scheme was significantly slower when increasing the time-step size, which is opposite what is the case when the L-scheme is applied to e.g.~the Richards equation (see \cite{illiano2021iterative}), where the non-linearity appears in the time derivative. It is important to emphasize that each L-scheme iteration is computationally cheap. It could be possible to optimize the L-scheme convergence by adjusting the value of $\mathcal L$ from time step to time step \cite{mitra2019lscheme}. The coupling iterations appeared to be insensitive with respect to the choice of $\mathcal L_\text{coup}$ and only slightly more iterations were needed when increasing the time-step size. Hence, the coupling iterations were robust and the resulting time-step restriction was not a limiting factor for the coupling iterations. 

In terms of geological applications like geothermal energy extraction, the formulated model and numerical scheme are suitable for modeling and simulation flow, transport and chemical reactions at the pore scale of the porous medium. Larger domains are usually modeled at the so-called Darcy scale, where only averaged quantities like porosity and averaged flow through permeability are used. The current phase-field model can be upscaled to Darcy scale by applying homogenization \cite{hornung1997homogenization}, similar as has been done for other phase-field models \cite{bringedal2019phase,redeker2016phase}. The advantage of upscaling phase-field model in contrast to sharp-interface models (as in \cite{noorden2009irc,bringedal2016upscaling}) is that constant domains are considered, which simplifies the upscaling steps. This however opens the question concerning asymptotic consistency; if the sharp-interface limit of the upscaled model is the same as the upscaled sharp-interface model, which has been established for a Cahn-Hilliard model \cite{lars2022upscaling}. To ensure a conservative Allen-Cahn model, the conservative Allen-Cahn equation could be applied to each so-called cell problem, which localizes the non-local term to smaller patches of the pore-scale domain. This opens for also efficient and conservative phase-field modeling using Allen-Cahn also in this case.

\section*{Acknowledgements}
We thank the Deutsche Forschungsgemeinschaft (DFG, German Research Foundation) for supporting this work by funding SFB 1313, Project Number 327154368.

\bibliographystyle{plain}
\bibliography{preprint_RAM}

\begin{thebibliography}{10}

\bibitem{allen1979cahn}
S.~M. Allen and J.~W. Cahn.
\newblock A microscopic theory for antiphase boundary motion and its
  application to antiphase domain coarsening.
\newblock {\em Acta Metallurgica}, 27(6):1085 -- 1095, 1979.

\bibitem{anderson1998diffuse}
Daniel~M Anderson, Geoffrey~B McFadden, and Adam~A Wheeler.
\newblock Diffuse-interface methods in fluid mechanics.
\newblock {\em Annual review of fluid mechanics}, 30(1):139--165, 1998.

\bibitem{BringedalBastidas3}
Manuela {Bastidas Olivares}, Carina Bringedal, and Iuliu~Sorin Pop.
\newblock A two-scale iterative scheme for a phase-field model for
  precipitation and dissolution in porous media.
\newblock {\em Applied Mathematics and Computation}, 396:125933, 2021.

\bibitem{beckermann1999melt}
C.~Beckermann, H.-J. Diepers, I.~Steinbach, A~Karma, and X.~Tong.
\newblock Modeling melt convection in phase-field simulations of
  solidification.
\newblock {\em Journal of Computational Physics}, 154(2):468 -- 496, 1999.

\bibitem{bringedal2016upscaling}
C.~Bringedal, I.~Berre, I.~S. Pop, and F.~A. Radu.
\newblock Upscaling of non-isothermal reactive porous media flow with changing
  porosity.
\newblock {\em Transport in Porous Media}, 114(2):371--393, 2016.

\bibitem{bringedal2019FVCA}
Carina Bringedal.
\newblock A conservative phase-field model for reactive transport.
\newblock In Robert Kl{\"o}fkorn, Eirik Keilegavlen, Florin~A. Radu, and
  J{\"u}rgen Fuhrmann, editors, {\em Finite Volumes for Complex Applications IX
  - Methods, Theoretical Aspects, Examples}, pages 537--545, Cham, 2020.
  Springer International Publishing.

\bibitem{Bringedal2015}
Carina Bringedal, Inga Berre, Iuliu~Sorin Pop, and Florin~Adrian Radu.
\newblock A model for non-isothermal flow and mineral precipitation and
  dissolution in a thin strip.
\newblock {\em Journal of Computational and Applied Mathematics}, 289:346 --
  355, 2015.
\newblock Sixth International Conference on Advanced Computational Methods in
  Engineering (ACOMEN 2014).

\bibitem{bringedal2019phase}
Carina Bringedal, Lars von Wolff, and Iuliu~Sorin Pop.
\newblock Phase field modeling of precipitation and dissolution processes in
  porous media: Upscaling and numerical experiments.
\newblock {\em Multiscale Modeling \& Simulation}, 18(2):1076--1112, 2020.

\bibitem{brun2020iterative}
Mats~Kirkes{\ae}ther Brun, Thomas Wick, Inga Berre, Jan~Martin Nordbotten, and
  Florin~Adrian Radu.
\newblock An iterative staggered scheme for phase field brittle fracture
  propagation with stabilizing parameters.
\newblock {\em Computer Methods in Applied Mechanics and Engineering},
  361:112752, 2020.

\bibitem{caginalp1988dynamics}
G.~Caginalp and P.~Fife.
\newblock Dynamics of layered interfaces arising from phase boundaries.
\newblock {\em SIAM Journal on Applied Mathematics}, 48(3):506--518, 1988.

\bibitem{cahn1958hilliard}
J.~W. Cahn and J.~E. Hilliard.
\newblock Free energy of a nonuniform system. i. interfacial free energy.
\newblock {\em The Journal of Chemical Physics}, 28(2):258--267, 1958.

\bibitem{Chang}
R.~Chang and K.~A. Goldsby.
\newblock {\em General Chemistry, The Essential Concepts}.
\newblock McGraw-Hill, 2014.

\bibitem{chen2010cons}
X.~Chen, D.~Hilhorst, and E.~Logak.
\newblock {Mass conserving Allen-Cahn equation and volume preserving mean
  curvature flow}.
\newblock {\em Interfaces and Free Boundaries}, 12:527--549, 2010.

\bibitem{clauser2003numerical}
Christoph Clauser.
\newblock {\em Numerical simulation of reactive flow in hot aquifers: SHEMAT
  and processing SHEMAT}.
\newblock Springer Science \& Business Media, 2003.

\bibitem{eymard2000finite}
Robert Eymard, Thierry Gallou{\"e}t, and Rapha{\`e}le Herbin.
\newblock {\em Finite volume methods}, volume~7.
\newblock Elsevier, 2000.

\bibitem{fasano2008RH}
Antonio Fasano.
\newblock Mathematical models of some diffusive processes with free boundaries.
\newblock {\em SIMAI e-Lecture Notes}, 1, 2008.

\bibitem{frank2020bound}
Florian Frank, Andreas Rupp, and Dmitri Kuzmin.
\newblock {Bound-preserving flux limiting schemes for DG discretizations of
  conservation laws with applications to the Cahn–Hilliard equation}.
\newblock {\em Computer Methods in Applied Mechanics and Engineering},
  359:112665, 2020.

\bibitem{garcke2015shape}
H.~Garcke, C.~Hecht, M.~Hinze, and C.~Kahle.
\newblock Numerical approximation of phase field based shape and topology
  optimization for fluids.
\newblock {\em SIAM Journal on Scientific Computing}, 37(4):A1846--A1871, 2015.

\bibitem{hornung1997homogenization}
Ulrich Hornung.
\newblock {\em Homogenization and porous media}, volume~6.
\newblock Springer Science \& Business Media, 1997.

\bibitem{illiano2021iterative}
Davide Illiano, Iuliu~Sorin Pop, and Florin~Adrian Radu.
\newblock Iterative schemes for surfactant transport in porous media.
\newblock {\em Computational Geosciences}, 25:805--822, 2021.

\bibitem{jacqmin2000contact}
D.~Jacqmin.
\newblock Contact-line dynamics of a diffuse fluid interface.
\newblock {\em Journal of Fluid Mechanics}, 402:57–88, 2000.

\bibitem{kelm2022comparison}
Mathis Kelm, Stephan G{\"a}rttner, Carina Bringedal, Bernd Flemisch, Peter
  Knabner, and Nadja Ray.
\newblock Comparison study of phase-field and level-set method for three-phase
  systems including two minerals.
\newblock {\em Computational Geosciences}, 26(3):545--570, 2022.

\bibitem{knabner1995compatible}
P.~Knabner, C.J. van Duijn, and S.~Hengst.
\newblock An analysis of crystal dissolution fronts in flows through porous
  media. part 1: Compatible boundary conditions.
\newblock {\em Advances in Water Resources}, 18(3):171 -- 185, 1995.

\bibitem{kumar2014convergence}
Kundan Kumar, Iuliu~Sorin Pop, and Florin~A Radu.
\newblock Convergence analysis for a conformal discretization of a model for
  precipitation and dissolution in porous media.
\newblock {\em Numerische Mathematik}, 127:715--749, 2014.

\bibitem{landau1987fluid}
L.~D. Landau and E.~M. Lifshitz.
\newblock {\em Fluid mechanics}, volume~6 of {\em Course of Theoretical
  Physics}.
\newblock Pergamon Press, 1987.
\newblock Translated from the Russian by J.~B.~Sykes and W.~H.~Reid.

\bibitem{li2008levelset}
Xiaoyi Li, Hai Huang, and Paul Meakin.
\newblock Level set simulation of coupled advection-diffusion and pore
  structure evolution due to mineral precipitation in porous media.
\newblock {\em Water Resources Research}, 44(12), 2008.

\bibitem{list2016iterative}
Florian List and Florin~A. Radu.
\newblock {A study on iterative methods for solving Richards' equation}.
\newblock {\em Computational Geosciences}, 20(2):341--353, Apr 2016.

\bibitem{mitra2019lscheme}
K.~Mitra and I.S. Pop.
\newblock A modified l-scheme to solve nonlinear diffusion problems.
\newblock {\em Computers \& Mathematics with Applications}, 77(6):1722--1738,
  2019.
\newblock 7th International Conference on Advanced Computational Methods in
  Engineering (ACOMEN 2017).

\bibitem{frank2018mass}
X.~Mu, F.~Frank, B.~Riviere, F.~O. Alpak, and W.~G. Chapman.
\newblock Mass-conserved density gradient theory model for nucleation process.
\newblock {\em Ind. Eng. Chem. Res.}, 57(48):16476--16485, 2018.

\bibitem{osher2005level}
Stanley Osher and Ronald~P Fedkiw.
\newblock {\em Level set methods and dynamic implicit surfaces}, volume~1.
\newblock Springer New York, 2005.

\bibitem{peszynska2021reduced}
Malgorzata Peszynska, Joseph Umhoefer, and Choah Shin.
\newblock Reduced model for properties of multiscale porous media with changing
  geometry.
\newblock {\em Computation}, 9(3):28, 2021.

\bibitem{plummer1988computer}
L~Niel Plummer.
\newblock {\em A computer program incorporating Pitzer's equations for
  calculation of geochemical reactions in brines}, volume~88.
\newblock Department of the Interior, US Geological Survey, 1988.

\bibitem{pop2004lscheme}
I.S. Pop, F.~Radu, and P.~Knabner.
\newblock {Mixed finite elements for the Richards’ equation: linearization
  procedure}.
\newblock {\em Journal of Computational and Applied Mathematics},
  168(1):365--373, 2004.
\newblock Selected Papers from the Second International Conference on Advanced
  Computational Methods in Engineering (ACOMEN 2002).

\bibitem{ray2019numerical}
N.~Ray, J.~Oberlander, and P.~Frolkovic.
\newblock Numerical investigation of a fully coupled micro-macro model for
  mineral dissolution and precipitation.
\newblock {\em Computational Geosciences}, Aug 2019.

\bibitem{redeker2016phase}
M.~Redeker, C.~Rohde, and I.~S. Pop.
\newblock Upscaling of a tri-phase phase-field model for precipitation in
  porous media.
\newblock {\em IMA Journal of Applied Mathematics}, 81(5):898--939, 2016.

\bibitem{lars2021ternary}
Christian Rohde and Lars von Wolff.
\newblock {A ternary Cahn–Hilliard–Navier–Stokes model for two-phase flow
  with precipitation and dissolution}.
\newblock {\em Mathematical Models and Methods in Applied Sciences},
  31(01):1--35, 2021.

\bibitem{rubinstein1992nonlocal}
J.~Rubinstein and P.~Sternberg.
\newblock {Nonlocal reaction—diffusion equations and nucleation}.
\newblock {\em IMA Journal of Applied Mathematics}, 48(3):249--264, 09 1992.

\bibitem{schlogl1972chemical}
F.~Schl{\"o}gl.
\newblock Chemical reaction models for non-equilibrium phase transitions.
\newblock {\em Zeitschrift f{\"u}r Physik}, 253(2):147--161, Apr 1972.

\bibitem{reactivereview}
Nicolas Seigneur, K.~Ulrich Mayer, and Carl~I. Steefel.
\newblock Reactive transport in evolving porous media.
\newblock {\em Reviews in Mineralogy and Geochemistry}, 85(1):197--238, 09
  2019.

\bibitem{stoth1996mullins}
Barbara~E.E. Stoth.
\newblock {Convergence of the Cahn–Hilliard equation to the Mullins–Sekerka
  problem in spherical symmetry}.
\newblock {\em Journal of Differential Equations}, 125(1):154--183, 1996.

\bibitem{noorden2009irc}
T.~L. van Noorden.
\newblock Crystal precipitation and dissolution in a porous medium: effective
  equations and numerical experiments.
\newblock {\em Multiscale Modeling \& Simulation}, 7(3):1220--1236, 2009.

\bibitem{vannoorden2009strip}
T.~L. van Noorden.
\newblock Crystal precipitation and dissolution in a thin strip.
\newblock {\em European Journal of Applied Mathematics}, 20(1):69–91, 2009.

\bibitem{lars2022upscaling}
Lars von Wolff and Iuliu~Sorin Pop.
\newblock {Upscaling of a Cahn–Hilliard Navier–Stokes model with
  precipitation and dissolution in a thin strip}.
\newblock {\em Journal of Fluid Mechanics}, 941:A49, 2022.

\bibitem{xu1999modeling}
Tianfu Xu, Javier Samper, Carlos Ayora, Marisol Manzano, and Emilio Custodio.
\newblock Modeling of non-isothermal multi-component reactive transport in
  field scale porous media flow systems.
\newblock {\em Journal of Hydrology}, 214(1):144--164, 1999.

\bibitem{xu2012compare}
Zhijie Xu, Hai Huang, Xiaoyi Li, and Paul Meakin.
\newblock Phase field and level set methods for modeling solute precipitation
  and/or dissolution.
\newblock {\em Computer Physics Communications}, 183(1):15--19, 2012.

\end{thebibliography}
\end{document}